\documentclass[a4paper, 12pt]{article}

\usepackage[top = 2.5cm, bottom = 2.5cm, left = 2.4cm, right = 2.4cm]{geometry}
\usepackage{amsfonts, amsmath, amsthm, amssymb, mathtools, cases, bbm, array}
\usepackage{lmodern, indentfirst, setspace, changepage, fancyhdr}
\usepackage{tikz, pgfplots, float, subcaption}
\usepackage[title]{appendix}
\usepackage[english]{babel}
\usepackage{cite}

\usepackage{hyperref}
\hypersetup{
	colorlinks = true,
	linkcolor = blue,
	citecolor = blue,
}

\newtheorem{proposition}{Proposition}
\newtheorem{corollary}{Corollary}
\newtheorem{theorem}{Theorem}
\newtheorem{lemma}{Lemma}
\theoremstyle{definition}

\newtheorem{condition}{Condition}
\newtheorem{definition}{Definition}

\newtheorem{remark}{Remark}

\newcommand{\argmax}[1]{\underset{#1}{\mathrm{argmax}}}
\newcommand{\ind}[1]{\mathbbm{1}_{\left\{#1\right\}}}
\newcommand{\norm}[1]{\left|\left|#1\right|\right|}

\newcommand{\ceil}[1]{\left\lceil#1\right\rceil}
\newcommand{\map}[3]{#1 : #2 \longrightarrow #3}
\newcommand{\set}[2]{\left\{#1 : #2\right\}}
\newcommand{\vct}[2]{\left(#1 : #2\right)}

\newcommand{\defeq}{\vcentcolon=}
\newcommand{\eqdef}{=\vcentcolon}

\newcommand{\bX}{\boldsymbol{X}}
\newcommand{\bY}{\boldsymbol{Y}}

\newcommand{\calA}{\mathcal{A}}
\newcommand{\calB}{\mathcal{B}}

\newcommand{\calM}{\mathcal{M}}
\newcommand{\calN}{\mathcal{N}}

\newcommand{\calS}{\mathcal{S}}

\newcommand{\N}{\mathbbm{N}}

\newcommand{\R}{\mathbbm{R}}

\newcommand{\e}{\mathrm{e}}

\newcommand{\expect}{E\expectarg}
\newcommand{\var}{\mathrm{Var}\expectarg}
\newcommand{\cov}{\mathrm{Cov}\expectarg}
\DeclarePairedDelimiterX{\expectarg}[1]{[}{]}{%
	\ifnum\currentgrouptype=16 \else\begingroup\fi
	\activatebar#1
	\ifnum\currentgrouptype=16 \else\endgroup\fi
}
\newcommand{\cprob}{P\probarg}
\DeclarePairedDelimiterX{\probarg}[1]{(}{)}{%
	\ifnum\currentgrouptype=16 \else\begingroup\fi
	\activatebar#1
	\ifnum\currentgrouptype=16 \else\endgroup\fi
}
\newcommand{\innermid}{\nonscript\;\delimsize\vert\nonscript\;}
\newcommand{\activatebar}{%
	\begingroup\lccode`\~=`\|
	\lowercase{\endgroup\let~}\innermid 
	\mathcode`|=\string"8000
}

\usetikzlibrary{automata, arrows, positioning, calc, external, babel, backgrounds, matrix, shapes}
\usepgfplotslibrary{fillbetween}
\pgfplotsset{
	compat = 1.16,
	ticklabel style = {font = \footnotesize},
	every axis/.append style = {
		grid style = {dashed, gray, opacity = 0.2},
		label style = {font = \footnotesize}, 
		width = \columnwidth,
		height = 0.618 * 1 * \columnwidth
	}
}

\definecolor{britishracinggreen}{rgb}{0.0, 0.26, 0.15}
\definecolor{bostonuniversityred}{rgb}{0.8, 0.0, 0.0}
\definecolor{ceruleanblue}{rgb}{0.16, 0.32, 0.75}
\definecolor{airforceblue}{rgb}{0.36, 0.54, 0.66}
\definecolor{cadmiumgreen}{rgb}{0.0, 0.42, 0.24}
\definecolor{ao(english)}{rgb}{0.0, 0.5, 0.0}
\definecolor{coolblack}{rgb}{0.0, 0.18, 0.39}
\definecolor{byzantine}{rgb}{0.74, 0.2, 0.64}
\definecolor{alizarin}{rgb}{0.82, 0.1, 0.26}
\definecolor{arsenic}{rgb}{0.23, 0.27, 0.29}
\definecolor{cobalt}{rgb}{0.0, 0.28, 0.67}
\definecolor{amber}{rgb}{1.0, 0.75, 0.0}

\title{Server saturation in skewed networks \vspace{\baselineskip}}

\author{
\begin{tabular}{ccc}
	\normalsize Diego Goldsztajn & \normalsize Sem C. Borst & \normalsize Johan S.H. van Leeuwaarden \\
	\scriptsize Inria & \scriptsize Eindhoven University of Technology & \scriptsize Tilburg University \\
	\scriptsize\texttt{diego.goldsztajn@inria.fr} & \scriptsize\texttt{s.c.borst@tue.nl} & \scriptsize\texttt{j.s.h.vanleeuwaarden@uvt.nl} \\
\end{tabular}
}

\date{\vspace{\baselineskip} April 9, 2024}

\begin{document}

	
\maketitle

\noindent\rule{\textwidth}{1pt}

\vspace{2\baselineskip}

\onehalfspacing

\begin{adjustwidth}{0.8cm}{0.8cm}
	\begin{center}
		\textbf{Abstract}
	\end{center}
	
	\vspace{0.3\baselineskip}
	
	\noindent	
	We consider a model inspired by compatibility constraints that arise between tasks and servers in data centers, cloud computing systems and content delivery networks. The constraints are represented by a bipartite graph or network that interconnects dispatchers with compatible servers. Each dispatcher receives tasks over time and sends every task to a compatible server with the least number of tasks, or to a server with the least number of tasks among $d$ compatible servers selected uniformly at random. We focus on networks where the neighborhood of at least one server is skewed in a limiting regime. This means that a diverging number of dispatchers are in the neighborhood which are each compatible with a uniformly bounded number of servers; thus, the degree of the central server approaches infinity while the degrees of many neighboring dispatchers remain bounded. We prove that each server with a skewed neighborhood saturates, in the sense that the mean number of tasks queueing in front of it in steady state approaches infinity. Paradoxically, this pathological behavior can even arise in random networks where nearly all the servers have at most one task in the limit.
	
	\vspace{\baselineskip}
	
	
	\vspace{0.3\baselineskip}
	
	\small{\noindent \textit{Acknowledgment:} supported by the Netherlands Organisation for Scientific Research (NWO) through Gravitation-grant NETWORKS-024.002.003 and Vici grant 202.068. Most of the research reported here was carried out while Diego Goldsztajn was affiliated with Eindhoven University of Technology.} 
\end{adjustwidth}

\newpage


\section{Introduction}
\label{sec: introduction}

Modern data centers and cloud computing systems operate under many compatibility constraints between tasks and servers. For example, image classification tasks require that a suitably trained deep convolutional neural network model has been loaded in the server, whereas natural language processing tasks require different machine learning models. In a similar way, content delivery networks offer many content items, like webpages, video or music, spreading them across multiple edge servers, so that each edge server provides quick access to a limited subset of items that are currently stored in its cache memory.

In this paper we model compatibility constraints between tasks and servers by means of bipartite graphs, as in \cite{rutten2022load,rutten2023meanSHORT,weng2020boptimal,zhao2022exploiting,zhao2023optimal}. Specifically, we consider networks of dispatchers and servers where each dispatcher is compatible with a subset of the servers, as specified by a bipartite graph. This model allows to describe complex networks, involving multiple clusters of servers that may be geographically distributed and hierarchically organized. In particular, the dispatchers need not represent actual load balancers, but rather correspond to streams of specific tasks that are distributed among common subsets of servers. For instance, the network could consist of geographically distributed clusters, each fed by a load balancer that aggregates nearby traffic. In this context, a dispatcher could correspond to the subset of servers in a given cluster that have a specific content item. Further, the latter subset could include servers in an overflow cluster that supports several other clusters. 

Complex network structures may lead to bottlenecks in task execution at specific points of the network. Such bottlenecks appear particularly at local network structures that we call skewed neighborhoods. Informally speaking, the neighborhood of a server is skewed if it contains many dispatchers such that each one is only compatible with a few servers. In this paper we prove that servers with skewed neighborhoods saturate, i.e., the number of tasks queueing in front of the server approaches infinity in a suitable limiting regime.

The heterogeneity of tasks, and  the specialization of certain servers, can play a key role in the formation of skewed neighborhoods. For example, image classification models can be trained on relatively specific sets of images or using broader data sets. A model trained in the former way is better suited for specialized tasks, such as classifying fruits, while models trained in the latter way are needed for more generic tasks, but may also perform some of the more specific tasks. This heterogeneity of tasks and servers can lead to skewed neighborhoods resembling the network depicted in Figure \ref{fig: dandelion network}, with boundary servers that specialize in certain tasks and central servers that might perform some of these tasks.

\begin{figure}
	\centering
	\begin{subfigure}{0.49\columnwidth}
		\centering
		\includegraphics{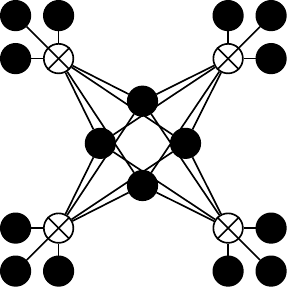}
	\end{subfigure}
	\hfill
	\begin{subfigure}{0.49\columnwidth}
		\centering
		{\raisebox{0mm}{%
		\includegraphics{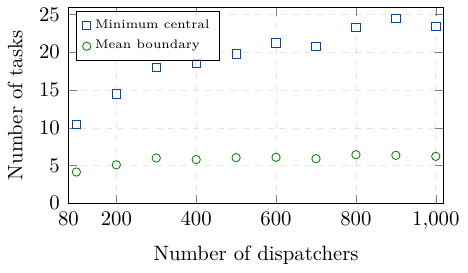}}}
	\end{subfigure}
	\caption{The schematic depicts a dandelion network. The dispatchers are represented by crossed white circles and the servers by black circles. The network has $b = 3$ boundary servers per dispatcher and $c = 4$ central servers; each dispatcher is compatible with all the central servers and with $b$ different boundary servers. The chart plots the mean number of tasks across the boundary servers, and the minimum number of tasks across the central servers, as the number of dispatchers increases while $b$ and $c$ remain constant. Tasks arrive to each dispatcher at rate $\lambda = 0.95 b$, and all the servers execute tasks at unit rate.}
	\label{fig: dandelion network}
\end{figure}

Similar heterogeneity and skewness properties can arise in content delivery networks; see \cite{nygren2010akamai}. In fact, we give a numerical example of the server saturation property inspired by such networks. Furthermore, we show that skewed neighborhoods can even arise in the absence of heterogeneity, in standard random graph models where the edge distribution is symmetric with respect to dispatchers and servers, respectively.

\subsection{Overview of the main result}
\label{sub: main contributions}

As alluded to earlier, the main result of this paper concerns the saturation of servers with skewed neighborhoods in a limiting regime. Formally speaking, the neighborhood of a server is skewed if it contains a diverging number of dispatchers such that each of these dispatchers is compatible with a uniformly bounded number of servers; the term skewed refers to the asymmetry between the diverging degree of the server and the uniformly bounded degrees of the dispatchers. We prove that each server with a skewed neighborhood saturates, with its mean number of tasks in steady state approaching infinity.

We focus on the dispatching rule whereby each dispatcher sends every incoming task to a compatible server with the least number of tasks. However, we also show that our main result holds if each dispatcher uses the following power-of-$d$ scheme: when a task arrives, the corresponding dispatcher samples $d$ compatible servers uniformly at random, and then sends the task to a sampled server with the least number of tasks. The result even extends to heterogeneous arrival and service rates, which allows to model task types with varying degrees of popularity and servers with different processing speeds.

\subsubsection*{Dandelion networks}

We first prove the saturation property for dandelion networks, as in Figure \ref{fig: dandelion network}, with homogeneous arrival and service rates, and then extend this result to general networks with some skewed neighborhoods. More specifically, we consider dandelion networks having a fixed number of central servers and a diverging number of dispatchers, with the number of boundary servers per dispatcher also fixed. We prove that the central servers saturate in the limit, and we show that the stationary behavior of the rest of the network is as if the central servers did not exist; i.e., the rest of the network behaves as a collection of isolated subnetworks, each having one dispatcher and the same number of servers.

The main intuition behind this result is easier to understand if we consider dandelion networks with one central server and one boundary server per dispatcher. All dispatchers can forward tasks to the central server, and will do so whenever the corresponding boundary server has more tasks than the central server. Therefore, the instantaneous rate at which tasks are forwarded to the central server is given by the number of boundary servers with a larger queue length. Nevertheless, the average rate at which the central server receives tasks must be equal to its average service completion rate in equilibrium, which cannot exceed its finite nominal service rate.  Hence, the average number of boundary servers with a larger queue length is limited by that nominal service rate, i.e., only a few boundary servers have a larger number of tasks in steady state. In other words, the number of tasks at the central server behaves roughly like a large order statistic for the number of tasks across the boundary servers. Intuitively, this suggests that the central server has a large number of tasks, which approaches infinity as the number of boundary servers grows. Moreover, if that is the case, then each boundary server behaves as if it was isolated, because its dispatcher hardly ever sends tasks to the saturated central server.


The above explanation relies on the special feature that there is just a single boundary server per dispatcher, and only one central server, but a more involved line of argumentation can be applied in the general setup. Formally speaking, the result for general dandelion networks is obtained by analyzing the generators of the Markov processes that describe the number of tasks at each server in a dandelion network and the network obtained by removing the central servers of the former network, respectively. More precisely, we compute the drifts of bounded functions with respect to these generators, and prove that their mean difference, with respect to the stationary distribution of the Markov process defined by the dandelion network, vanishes in the limit. We conclude that both processes have the same stationary distribution in the limit by invoking a general result from the theory of Markov processes, which characterizes the stationary distribution of a Markov process in terms of the steady-state drift of a suitable class of functions. This establishes the asymptotic independence of the sets of boundary servers associated with the same dispatcher. The saturation of the central servers is then obtained by establishing that in steady state nearly all the boundary servers have fewer tasks on average, in a similar way as above, and by leveraging the asymptotic independence property.

Drift analysis techniques have been used to derive fluid and diffusion approximations for a wide range of queueing systems; see \cite{ying2017stein,braverman2016stein,gurvich2014diffusion,gast2017refined}. Also, the so-called drift method has been used in \cite{eryilmaz2012asymptotically,maguluri2016heavy,wang2018heavy,liu2020steady,weng2020optimal} to obtain asymptotic bounds for steady-state averages of various performance metrics. The latter methodologies rely on analyzing the drifts of specific functions, solutions of Poisson or Stein equations, or suitable Lyapunov functions. In contrast, the technique developed in this paper involves considering the drifts of generic bounded functions, and has not been used before to the best of our knowledge.

\subsubsection*{General networks}

We extend the above result for dandelion networks to general networks with skewed neighborhoods by using a novel stochastic coupling. Namely, we define monotone network transformations with the following property: if the servers in the original and transformed network have the same number of tasks at time zero, then each server in the transformed network has fewer tasks than the same server in the original network at all times.

The key transformation allows to remove a compatibility relation between any given server and dispatcher, by attaching a new server to the dispatcher. Intuitively, this should reduce the load of the former server without increasing the load of the other servers in the original network. Nonetheless, we need to synchronize the departures from the given server and the newly added server in order to obtain stochastic lower bounds.
The other transformations that we use are adding a server to a dispatcher, decreasing the arrival rate of tasks at some dispatchers and increasing the processing speeds of some servers; the latter two allow to handle heterogeneous arrival and service rates.

By suitably combining the above transformations, we can transform any network where a server has a skewed neighborhood into a network with an isolated dandelion subnetwork, with homogeneous arrival and service rates. Moreover, the transformations can be applied in such a way that the number of tasks at the latter server is stochastically lower bounded by the number of tasks at a central server of the dandelion subnetwork, while the number of dispatchers in the dandelion subnetwork approaches infinity. Then our result for dandelion networks implies that the server with a skewed neighborhood saturates in the limit.


\subsection{Several examples}
\label{sub: several examples}

Networks with skewed neighborhoods can arise in a variety of ways. For example, consider any sequence of networks where the maximum degree across the dispatchers is bounded. Now suppose that a fixed number of servers are attached to each network, in an arbitrary manner such that each server is connected to a diverging number of dispatchers. Since a fixed number of servers are attached, the degrees of the dispatchers remain bounded, and thus each of the attached servers has a skewed neighborhood. Conversely, starting from a sequence of dense networks, skewed neighborhoods can be generated by removing edges. For instance, we may select some servers whose neighborhoods will be skewed, and with some suitably selected probability, remove each of the nonadjacent edges.

The former construction can be generalized in several ways. In particular, the maximum degree across the dispatchers need not be bounded; it suffices that the degrees of a diverging number of dispatchers remain bounded, and that the neighborhoods of the attached servers contain diverging numbers of these dispatchers. Also, it is possible to obtain a diverging number of skewed neighborhoods by attaching a diverging number of servers. For example, this can be achieved by attaching the servers to disjoint sets of dispatchers with bounded degrees, such that the sizes of these sets diverge. In fact, it is enough that the neighborhoods of these servers do not overlap too much, so that each attached server is connected to a diverging number of dispatchers that retain the bounded degree property.

As noted earlier, skewed neighborhoods can also arise in content delivery networks, and in this paper we simulate a network inspired by this application. This network has many edge servers, distributed across multiple clusters, and some origin servers located in a supporting cluster. Each of multiple popular content items is replicated at an appropriate number of edge servers in each cluster and is also available at the origin servers, whereas unpopular items are only stored at the latter servers. This structure can be captured by our model if we consider one virtual dispatcher for each cluster and content item. The number of dispatchers equals the number of clusters times the number of items, and the degree of each dispatcher is the sum of the number of replicas per cluster and the number of origin servers. Loosely speaking, if the former number is large with respect to the latter number, then the neighborhoods of the origin servers are skewed, and our simulations show a striking imbalance between the mean queue lengths of edge and origin servers.

The existence of skewed neighborhoods is favored by task heterogeneity and server specialization, as in the latter example. Nevertheless, we show that skewed neighborhoods can even arise in highly symmetric random setups. In particular, we provide two examples based on standard random graph models where the edge distribution is symmetric with respect to dispatchers and servers, respectively. In the first example we consider random bipartite graphs where the edges are drawn independently with the same probability, and the mean degree remains constant as the size of the network grows. The second example involves Erd\H{o}s-R\'enyi random graphs, where instead the mean degree diverges slowly.

Paradoxically, a result proved in \cite{mukherjee2018asymptotically} implies that the fraction of servers with more than one task vanishes as the networks in the second example grow. Hence, the limiting mean performance is excellent but the steady-state number of tasks in at least one server approaches infinity. This suggests that server saturation could go unnoticed. However, like the central servers in a dandelion network, the saturated servers hardly ever receive tasks from any given dispatcher, and thus contribute only negligibly to the performance of other servers. This means that servers with skewed neighborhoods can be removed without degrading performance, releasing valuable computing resources.

\subsection{Related work}
\label{sub: related work}

Immense attention has been directed to the classical load balancing model where a centralized dispatcher assigns incoming tasks to the servers; see \cite{van2018scalable} for an extensive survey. It was proved in \cite{menich1991optimality,sparaggis1993extremal} that the Join the Shortest Queue (JSQ) policy is optimal in a stochastic majorization sense. This algorithm has scalability issues though, which have sparked a strong interest in alternative policies with lower implementation overhead. For instance, power-of-$d$ schemes that assign every task to the shortest of $d$ queues selected uniformly at random, which were first examined in \cite{mitzenmacher2001power,vvedenskaya1996queueing}. While these policies can be deployed in large systems, they do not enjoy the same optimality properties as JSQ.

The problem of balancing a fixed workload across the nodes of a static graph was first studied in \cite{cybenko1989dynamic}, and the situation where the graph is dynamic was first considered in \cite{bahi2003broken,elsasser2004load}; for a more extensive list of references see \cite{gilbert2021complexity}. Most of this literature assumes not only that the workload is fixed but also that the graph or the sequence of graphs that describes the network is deterministic or adversarial. The main goal is to design algorithms that converge to a state of uniformly balanced workload under different conditions on the graphs, and to analyze their complexity. Load balancing on static graphs has also been studied in the balls-and-bins context where balls arrive to the nodes of the graph and simply accumulate; e.g., see \cite{kenthapadi2006balanced,wieder2017hashing}. This situation is also fundamentally different from the queueing scenario considered in the present paper. For example, in the balls-and-bins setup the total number of balls at any given time is independent of the way in which the balls are assigned to the bins, which is not the case when the balls are replaced by tasks and the bins by servers that execute the tasks. In addition, a round-robin assignment perfectly balances the allocation of balls to bins but is far from optimal in the queueing setup. 

The first papers to study load balancing on static graphs from a queueing perspective are \cite{gast2015power,turner1998effect}, which focus particularly on ring topologies. These papers establish that the flexibility to forward tasks to a few neighbors substantially improves performance in terms of the waiting time. Nevertheless, they also show via numerical results that performance is sensitive to the graph, and that the possibility of forwarding tasks to a fixed set of $d - 1$ neighbors does not match the performance of classical power-of-$d$ schemes in the complete graph case. The results in \cite{mukherjee2018asymptotically} provide connectivity conditions on the graph for achieving asymptotic fluid and diffusion optimality when tasks can be forwarded to any neighboring server. Dynamic graphs are studied in \cite{goldsztajn2023sparse}, where the fluid limit of the occupancy process is derived for graph topologies that may have uniformly bounded degrees. The situation where tasks can only be forwarded to a uniformly selected random subset of $d$ neighbors was studied in \cite{budhiraja2019supermarket}, which considers static graphs and provides connectivity conditions for obtaining the same fluid limit as for the sequence of complete graphs. In a different line of research, \cite{tang2019random} analyzed power-of-$d$ algorithms that involve less randomness by using nonbacktracking random walks on a high-girth graph to sample the servers.

In the past few years, several papers have considered the situation where tasks may be of different types and servers may only be compatible with certain types. As in the present paper, the predominant model has been to replace the graph interconnecting the servers by a bipartite graph between task types and servers, for specifying which task types can be executed by each server. A different but related model replaces these strong compatibility constraints with soft affinity relations, which imply that every server can execute all task types, but at a service rate that depends on the affinity between the server and the task; we refer to \cite{cardinaels2019job,xie2015priority,xie2016scheduling,wang2014maptask} for this different stream of literature.

For the model with strict compatibility constraints, general stability conditions are provided in \cite{cruise2020stability,bramson2011stability,foss1998stability}. In addition, \cite{rutten2022load} assumes that every new task joins the least busy of $d$ compatible servers chosen uniformly at random, and provides connectivity conditions such that the occupancy process has the same process-level and steady-state fluid limit as in the case where the graph is complete bipartite. Similar models are considered in \cite{rutten2023meanSHORT,zhao2022exploiting}. The former of these papers broadens the class of graph sequences for which the steady-state fluid limit in \cite{rutten2022load} holds. For instance, \cite{rutten2023mean,rutten2023meanSHORT} considers certain sequences of spatial graphs that do not satisfy the strong connectivity conditions stated in \cite{rutten2022load}; yet the mean number of servers compatible with any task type goes to infinity. On the other hand, \cite{zhao2022exploiting} extends the model studied in \cite{rutten2022load} by allowing for heterogeneous service rates, and proves process-level and steady-state fluid limits in this setting. The model considered in \cite{weng2020optimal,weng2020boptimal} also allows for heterogeneous service rates, and shows that two specific policies are asymptotically optimal with respect to the mean response time of tasks, provided that suitable connectivity conditions hold. The service rate of a task is further allowed to depend on both the task type and the assigned server in \cite{zhao2023optimal}, which proves that two other load balancing policies are fluid optimal if certain connectivity and subcriticality conditions hold.

All of the above papers consider sequences of networks such that the average degree approaches infinity and stronger connectivity conditions hold. In contrast, this paper concerns sequences of networks with skewed degrees, where the degrees of some servers approach infinity but many dispatchers have uniformly bounded degrees. Moreover, our proof techniques are quite different from the process-level arguments used in \cite{rutten2022load,rutten2023meanSHORT,rutten2023mean,zhao2022exploiting,zhao2023optimal}, and also from the drift method used in \cite{weng2020optimal,weng2020boptimal}. Our proofs also rely on drift analysis, but the latter papers focus on the drifts of suitable Lyapunov functions, or solutions to Stein equations, whereas we examine asymptotic properties of the drifts of general functions.

\subsection{Organization of the paper}
\label{sub: organization of the paper}

The remainder of the paper is organized as follows. In Section \ref{sec: model description} we define the model formally. In Section \ref{sec: main result} we state the main result and comment on future research directions. In Section \ref{sec: simulation experiment} we illustrate the server saturation property through a simulation example inspired by content delivery networks. In Section \ref{sec: drift analysis} we introduce notation pertaining to the drift of functions with respect to the generator of a continuous-time Markov chain, and we discuss some useful properties. In Section \ref{sec: dandelion networks} we prove the main result for dandelion networks. In Section \ref{sec: monotone transformations} we define the monotone network transformations and prove the stochastic coupling results. In Section \ref{sec: proof of the main result} we establish the main result for more general networks. In Section \ref{sec: random networks} we provide two examples using randomly generated networks, including that where nearly all the servers have at most one task in the limit and yet server saturation occurs. Some intermediate results are proved in Appendices \ref{app: proofs of various results} and \ref{app: lemmas used in the examples}.

\section{Model description}
\label{sec: model description}

We consider networks represented by bipartite graphs $G = (D, S, E)$. The finite sets $D$ and $S$ represent dispatchers and servers, respectively, and compatibility constraints are encoded in $E \subset D \times S$. Namely, dispatcher $d$ can only send tasks to server $u$ if $(d, u) \in E$. Data locality can create compatibility constraints of this kind in data centers or cloud computing systems. Indeed, dispatchers that receive tasks which require specific data to be processed typically assign these tasks to servers where this data has been prestored. Similarly, executing certain tasks may involve different machine learning models that have been specifically trained, and not all these models are available at each of the servers.

The policy used at the dispatchers to assign the incoming tasks to the servers can have a critical impact on the waiting times of tasks and server utilization. We focus on a natural policy that assigns each task to a server selected uniformly at random among the compatible servers with the least number of tasks; we assume that this scheme is used unless it is explicitly stated otherwise. However, we also prove our main result for the following power-of-$d$ policy: when a task arrives, the dispatcher samples $d$ compatible servers uniformly at random and sends the task to a sampled server with the least number of tasks, with ties being broken uniformly at random. If the processing speeds of the servers are known by the dispatchers, then it makes more sense to assign each task to the fastest compatible server with the least number of tasks, instead of breaking ties at random. However, information about the processing speeds may not be available.

We assume that each dispatcher $d$ receives tasks as an independent Poisson process of intensity $\lambda(d)$. In addition, tasks are executed sequentially at each server $u$ and have independent and exponentially distributed service times with rate $\mu(u)$. If we let $\bX(t, u)$ be the number of tasks in server $u$ at time $t$, then the latter assumptions imply that the stochastic process $\bX$ is a continuous-time Markov chain with values in $\N^S$.

\begin{definition}
	\label{def: load balancing process associated with bipartite graph}
	We say that $\bX$ is the load balancing process associated with the bipartite graph $G = (D, S, E)$ and the rate functions $\map{\lambda}{D}{(0, \infty)}$ and $\map{\mu}{S}{(0, \infty)}$.
\end{definition}

We are interested in the stationary behavior of load balancing processes. A sufficient condition for the ergodicity of such processes can be obtained from \cite[Theorem 2.5]{foss1998stability}. Namely, let $\bX$ be as before and let $\calN(d) \defeq \set{u \in S}{(d, u) \in E}$ denote the set of servers that are compatible with some dispatcher $d$. Then the condition 
\begin{equation}
	\label{eq: ergodicity condition}
	\sum_{\calN(d) \subset U} \lambda(d) < \sum_{u \in U} \mu(u) \quad \text{for all} \quad \emptyset \neq U \subset S
\end{equation}
implies that $\bX$ is ergodic. Conversely, suppose that there exists $U \subset S$ such that the strict inequality holds in the opposite direction. Then \cite[Theorem 2.7]{foss1998stability} implies that the process $\bX$ is not ergodic. Moreover, if $\set{\tau_k}{k \geq 1}$ denote the arrival times of tasks to the system, then the following instability property holds with probability one:
\begin{equation*}
	\liminf_{k \to \infty} \frac{1}{k}\sum_{u \in U} \bX(\tau_k, u) > 0.
\end{equation*}

\section{Main result}
\label{sec: main result}

Consider a sequence $\set{\bX_n}{n \geq 1}$ of load balancing processes defined by some networks $G_n = (D_n, S_n, E_n)$ and some rate functions $\map{\lambda_n}{D_n}{(0, \infty)}$ and $\map{\mu_n}{S_n}{(0, \infty)}$.
The neighborhoods of a dispatcher $d$ and a server $u$ are defined by
\begin{equation*}
\calN_n(d) \defeq \set{v \in S_n}{(d, v) \in E_n} \quad \text{and} \quad \calN_n(u) \defeq \set{e \in D_n}{(e, u) \in E_n},
\end{equation*}
respectively. The former set contains all the servers that are compatible with $d$, and the latter set contains all the dispatchers that are compatible with $u$. The sizes of these sets are denoted by $\deg_n(d)$ and $\deg_n(u)$, respectively; we refer to these quantities as the degrees of $d$ and $u$. The main result of this paper applies when the sequence of networks has at least one skewed neighborhood, in the sense of the following definition.

\begin{definition}
	\label{def: unevenly connected neighborhoods}
	Fix $\alpha = (a, \lambda_{\min}, \mu_{\max}) \in \N \times (0, \infty) \times (0, \infty)$. For each $u \in S_n$, let
	\begin{equation*}
	\calN_n^\alpha(u) \defeq \set{d \in \calN_n(u)}{\deg_n(d) \leq a, \lambda_n(d) \geq \lambda_{\min}\ \text{and}\ \mu_n(v) \leq \mu_{\max}\ \text{for all}\ v \in \calN_n(d)}.
	\end{equation*}
	We say that the servers $\set{u_n \in S_n}{n \geq 1}$ have a skewed neighborhood if
	\begin{equation}
	\label{eq: main condition for main result}
	\lim_{n \to \infty} \left|\calN_n^\alpha(u_n)\right| = \infty
	\end{equation}
	for some fixed $\alpha \in \N \times (0, \infty) \times (0, \infty)$.
\end{definition}

The sequence of networks under consideration could represent how a real network of dispatchers and servers scales as new task types are introduced and demand changes. In this case the notation $u_n$ could represent the same server for each $n$. However, the definition is more general and covers networks that are generated independently from each other.

Note that the conditions involving $\lambda_{\min}$ and $\mu_{\max}$ hold automatically if the arrival rates of tasks at all the dispatchers of the system remain uniformly lower bounded as $n \to \infty$ and the processing speeds of all the servers are uniformly upper bounded. These conditions just ensure that the arrival rate of tasks does not vanish in the limit for some dispatcher, and that the processing speed of some server does not approach infinity with $n$.

On the other hand, the condition involving $a$ refers to the network structure around the servers $u_n$. It says that $u_n$ must be compatible with a diverging number of dispatchers having uniformly bounded degrees; the servers $u_n$ can at the same time be compatible with any number of dispatchers with degrees that approach infinity. We say that the neighborhood $\calN_n(u_n)$ is skewed because $\deg_n(u_n) \to \infty$ as $n \to \infty$ while $\calN_n(u_n)$ contains dispatchers $d$ such that $\deg_n(d)$ remains uniformly bounded.

The main result of this paper is that servers with skewed neighborhoods saturate in the limit. This property is formally stated in the following theorem.

\begin{theorem}
	\label{the: main result}
	Suppose that there exists a sequence of servers $\set{u_n \in S_n}{n \geq 1}$ with a skewed neighborhood, in the sense of Definition \ref{def: unevenly connected neighborhoods}. Then the servers $u_n$ saturate with tasks as $n \to \infty$, in the following sense:
	\begin{equation}
	\label{eq: weak limit is infinity}
	\lim_{n \to \infty} \liminf_{t \to \infty} P\left(\bX_n(t, u_n) \geq k\right) = 1 \quad \text{for all} \quad k \in \N.
	\end{equation}
\end{theorem}

It is straightforward to check that \eqref{eq: weak limit is infinity} implies that
\begin{equation*}
\lim_{n \to \infty} \liminf_{t \to \infty} E\left[\bX_n(t, u_n)\right] = \lim_{n \to \infty} \liminf_{t \to \infty} \sum_{k = 1}^\infty P\left(\bX_n(t, u_n) \geq k\right) = \infty.
\end{equation*}
Moreover, suppose that the load balancing processes $\bX_n$ are ergodic and let $X_n$ have the stationary distribution of $\bX_n$ for each $n$. Then it is immediate that
\begin{equation*}
	\lim_{n \to \infty} P\left(X_n(u_n) \geq k\right) = 1 \quad \text{for all} \quad k \in \N \quad \text{and} \quad \lim_{n \to \infty} E\left[X_n(u_n)\right] = \infty.
\end{equation*}

\begin{remark}
	\label{rem: saturation of infinitely many servers}
	The theorem refers to one sequence $\set{u_n \in S_n}{n \geq 1}$, but the result holds for any number of sequences of servers, even infinitely many. If the conditions of the theorem hold simultaneously for distinct sequences, then each sequence of servers saturates.
\end{remark}

The following corollary shows that the saturation of servers with skewed neighborhoods also occurs when the dispatchers apply a power-of-$d$ policy.

\begin{corollary}
	\label{cor: power of d}
	Suppose that the dispatchers use a power-of-$d$ scheme and there is a sequence of servers $\set{u_n \in S_n}{n \geq 1}$ with skewed neighborhoods. Then \eqref{eq: weak limit is infinity} holds.
\end{corollary}

Corollary \ref{cor: power of d} is proved in Appendix \ref{app: proofs of various results}, by observing that networks where a power-of-$d$ scheme is applied are equivalent to suitably defined networks where each task is assigned to a server with the least number of tasks among all the compatible servers. Theorem \ref{the: main result} is proved in Sections \ref{sec: drift analysis}-\ref{sec: proof of the main result}. As noted in Section \ref{sub: main contributions}, we first prove the theorem for dandelion networks, and then extend it to general networks using stochastic coupling.

The main result of the present paper reveals how compatibility constraints between tasks and servers can create bottlenecks in networked service systems, and identifies local network structures where these bottlenecks arise; specifically, skewed neighborhoods that are particularly favored by task heterogeneity and server specialization. Since this result is asymptotic in nature, it cannot be used to quantify the extent of the bottlenecks in finite systems, i.e., the degree of server saturation. Nonetheless, it would be worth exploring server saturation in nonasymptotic setups for specific skewed neighborhood structures, such as dandelion structures, which possess a substantial degree of symmetry.

Clearly, it is not network structure alone that governs the performance of networked service systems, but rather the combination of network structure and the load balancing policy used by the dispatchers. The two policies considered in this paper are known to perform well in networks that are symmetric or have suitably large neighborhoods; e.g., dispatching every incoming task to a compatible server with the shortest queue is optimal in complete bipartite graphs and is asymptotically optimal under suitable connectivity conditions. However, egalitarian and network-oblivious policies of this kind may not be the best option for networks with strong asymmetries like skewed neighborhoods. In such asymmetric environments, it seems reasonable that servers with different neighborhood characteristics are treated differently. Designing and analyzing such network-aware load balancing policies is an interesting direction for future research.

\section{Simulation experiment}
\label{sec: simulation experiment}

Next we illustrate skewed neighborhoods and server saturation through an example inspired by content delivery networks. From a high-level perspective, such networks consist of many geographically distributed clusters of edge servers and a few supporting clusters of origin servers. Replicas of relatively popular content items are suitably replicated across edge servers, whereas all items, even hardly popular ones, are available at the origin servers. The clusters at the edge are fed by load balancers that aggregate content requests from nearby sources, and they can forward these requests to the origin servers. The diagram on the left of Figure~\ref{fig: cdn} provides a toy example of this network structure where each edge server has two content items and not all the edge servers have the same set of items.

\begin{figure}
	\begin{subfigure}{0.49\columnwidth}
		\centering
		\includegraphics{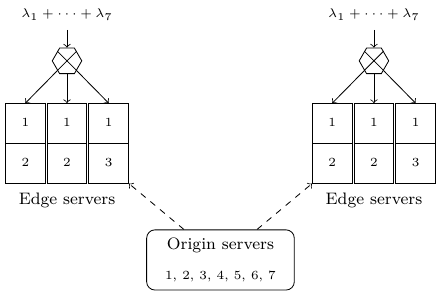}
	\end{subfigure}
	\hfill
	\begin{subfigure}{0.49\columnwidth}
		\centering
		\includegraphics{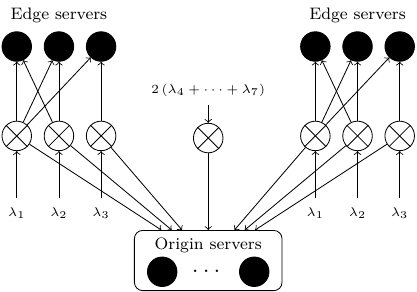}
	\end{subfigure}
	\caption{The diagram on the left illustrates the structure of a content delivery network. Crossed hexagons represent load balancers, numbers correspond to different content items and $\lambda_i$ denotes the arrival rate of requests for item $i$; the squares inside the edge servers represent memory slots and indicate which items are available at each server. In the schematic on the right, the load balancers have been replaced by crossed circles representing virtual dispatchers, and servers are depicted using black circles.}
	\label{fig: cdn}
\end{figure}

The model introduced in Section \ref{sec: model description} can capture the compatibility constraints between content items and servers. For this purpose we replace the load balancers in the latter diagram by virtual dispatchers as in the schematic on the right of Figure \ref{fig: cdn}. Specifically, each cluster and particular content item present in the cluster define a virtual dispatcher which distributes the requests associated with the item across the edge servers of the cluster that have the item and the origin servers; each request is assigned to the server with the least number of pending requests among the latter servers. There is an additional dispatcher that assigns the requests of content items not available at the edge servers to the origin servers. The arrival rate of requests at a virtual dispatcher depends on the content item associated with the dispatcher, and in particular on the popularity of the item.

\begin{remark}
	The network structure described above is a stylized model of a content delivery network with several abstractions and idealizing assumptions; see \cite{nygren2010akamai}. The main goal is not to model or evaluate the performance of these networks in any detailed way, but rather to capture some of their most salient features and exemplify how skewed neighborhoods can arise due to a combination of prevalent (power-law) popularity statistics, tiered network architectures and replication strategies. Some of the simplifications that we make are as follows. First, content items stored at the edge of a real network change over time, e.g., according to some caching policy. Instead, we will assume here that the items stored in the edge servers are fixed and that the number of replicas per cluster reflects the popularity of the items. Second, content requests arriving to a given cluster do not reach the origin servers unless the content item is not available at the cluster or the workload of the edge servers having the item exceeds a predefined threshold; edge servers are usually closer in terms of round trip time and are thus preferred. In contrast, the dispatchers in our model do not distinguish between edge and origin servers since they dispatch the incoming requests based only on the number of pending requests per server. However, all the origin servers have at least $7$ pending requests more than $99.9\%$ of the time for the simulation shown in Figure \ref{fig: simulation}. The behavior of the simulated system in steady state would be the same if we introduced a high-workload threshold corresponding to $7$ or fewer pending requests. 
\end{remark}

The plot shown in Figure \ref{fig: simulation} corresponds to a network with the structure described in Figure \ref{fig: cdn} but with a much larger size and higher complexity. Instead of a few individual content items, as in the toy example, we now consider a large catalog of items, which we arrange in order of popularity and divide into four tiers. The first tier $T_1$ contains the $0.1\%$ most popular items, while the remaining three tiers $T_2$, $T_3$ and $T_4$ contain the subsequent $1\%$, $10\%$ and $88.9\%$ of the items, respectively. Tiers are split into subsets of size $0.1\%$ so that the aggregate arrival rate of requests for items in a subset is constant across the subsets $T_{i, j}$ of a given tier $T_i$. The subsets are stored at the edge servers of a cluster as illustrated by the schematic of Figure \ref{fig: simulation}; e.g., each content item in $T_1$ has replicas in all the edge servers and items in $T_4$ are only available at the origin servers. We assume that the popularity of the content items follows a power-law  distribution such that the aggregate arrival rates for items in a given tier is equal for $T_1$, $T_2$ and $T_3$.

\begin{figure}
	\begin{subfigure}{0.49\columnwidth}
		\centering
		\includegraphics{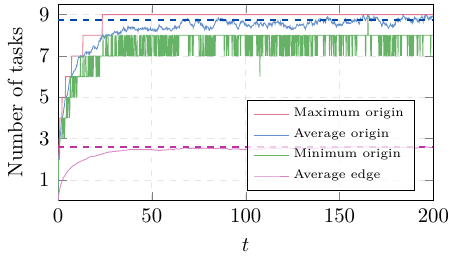}
	\end{subfigure}
	\hfill
	\begin{subfigure}{0.49\columnwidth}
		\centering
		\includegraphics{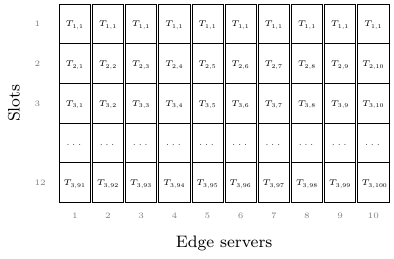}
	\end{subfigure}
	\caption{The chart plots the average number of tasks across the edge servers, as well as the minimum, average and maximum number of tasks across the origin servers; the dashed lines correspond to time averages over the interval $[150, 200]$. The schematic illustrates the distribution of replicas across the edge servers of a cluster. Here $T_{i, j}$ refers to replicas of the items in the $j$th subset of tier $T_i$, and $T_{1, 1} = T_1$.}
	\label{fig: simulation}
\end{figure}

The simulation in Figure \ref{fig: simulation} corresponds to a network having $1000$ clusters with $10$ edge servers each, supported by a single cluster of $100$ origin servers with the same processing speed as the edge servers. In order to assess the burden imposed by the edge servers on the origin servers, we assume that the arrival rate for items in $T_4$ is zero, i.e., every incoming request can be served by the edge servers. Moreover, we assume that the total arrival rate of content requests is $90\%$ of the combined processing speed of all the edge servers. Thus, each cluster alone would be stable in the absence of the origin servers.

The content placement in Figure \ref{fig: cdn} can be modeled with $11$ virtual dispatchers per cluster. Specifically, one dispatcher distributes requests for items in $T_{1, 1} = T_1$ across the $10$ edge servers in the cluster. In addition, there is a dispatcher that sends requests for items in $T_{2,i} \cup T_{3, i} \cup T_{3, i + 10} \cup \dots \cup T_{3, i + 90}$ to edge server $i$ for each $i \in \{1, \dots, 10\}$. The latter $10$ dispatchers receive requests at the same rate, whereas the former dispatcher receives requests at a five times higher rate. This is the coarser model that we can consider, but models with more virtual dispatchers are possible as well, e.g., with one dispatcher per subset $T_{i, j}$. Considering all the clusters, the total number of virtual dispatchers is $11000$ and each dispatcher is compatible with all the origin servers and at most $10$ edge servers. Thus, the degree of any virtual dispatcher is at most $110$, and the neighborhoods of the origin servers are skewed, in the sense that the number of dispatchers in the neighborhood is much larger than their maximum degree. In addition, if the dispatchers associated with tiers $T_2$ and $T_3$ were removed, then the network would be a dandelion with $100$ central servers, $1000$ dispatchers and $10$ boundary servers per dispatcher. On the other hand, if the dispatchers associated with tier $T_1$ were removed, then we would obtain a dandelion with $10000$ dispatchers and $1$ boundary server per dispatcher. Loosely speaking, the overall structure is a hybrid of the latter structures, weighed by the associated arrival rates.

The asymptotic result in Theorem \ref{the: main result} carries over to this finite scenario, where Figure \ref{fig: simulation} shows that the mean number of pending requests across the origin servers is rather large in steady state, more than three times larger than the mean across the edge servers. Also, the minimum number of requests across the origin servers remains larger than or equal to $7$ more than $99.9\%$ of the time, which nearly triples the average across the edge servers. Moreover, the average fraction of origin servers with strictly fewer than $8$ requests in steady state is less than $0.5\%$, and all the other servers have $8$ or $9$ requests. Requests forwarded to the cluster of origin servers are distributed across the servers using the JSQ policy, and the latter averages indicate that the cluster has a very high load.

While the plot in Figure \ref{fig: simulation} cannot be directly explained by Theorem \ref{the: main result}, the intuitive arguments in Section \ref{sub: main contributions} carry over, particularly in view of the resemblance to dandelion networks noted earlier. Informally speaking, the edge behaves as a collection of relatively small and weakly correlated subsystems due to the compatibility constraints; each of these subsystems is analogous to a set of boundary servers compatible with the same dispatcher in a dandelion network. Also, the minimum number of pending requests across the origin servers behaves as a large order statistic for the minimum number of requests across the subsystems; otherwise tasks arrive to the origin servers at a higher rate than they can sustain, as explained in Section \ref{sub: main contributions}. As a result, the minimum number of requests across the origin servers is much larger than the average across the edge servers.

\section{Drift analysis}
\label{sec: drift analysis} 

Our results for the dandelion networks are proved by analyzing the drifts of functions with respect to the generator of the corresponding load balancing process. In this section we define the latter drifts and we state a few important properties.

Let $\bX$ be the load balancing process associated with the bipartite graph $G = (D, S, E)$ and the rate functions $\map{\lambda}{D}{(0, \infty)}$ and $\map{\mu}{S}{(0, \infty)}$. The generator of the continuous-time Markov chain $\bX$ is given by the rate matrix $A$ such that $A(x, y)$ is the transition rate from $x$ to $y$. The drift of $\map{f}{\N^S}{\R}$ is the function defined by
\begin{equation*}
Af(x) \defeq \sum_{y \neq x} A(x, y) \left[f(y) - f(x)\right] \quad \text{for all} \quad x \in \N^S.
\end{equation*}
If $f$ is a bounded function, then the right-hand side is finite since the sum of $A(x, y)$ over all $y \neq x$ is upper bouned by the sum of all the arrival and service rates; and in fact $A$ is a bounded operator on the space of bounded functions endowed with the uniform norm. Nonetheless, we will also consider the drifts of some unbounded functions.

In order to provide a more explicit expression for $Af(x)$, it is convenient to introduce some additional notation. In particular, define
\begin{equation*}
\calM(d, x) \defeq \mathrm{argmin} \set{x(u)}{u \in \calN(d)} \quad \text{for all} \quad d \in D \quad \text{and} \quad x \in \N^S.
\end{equation*}
Recall that $\calN(d)$ is the set of servers that are compatible with $d$. Thus, $\calM(d, \bX)$ is the set of servers with the least number of tasks among those compatible with dispatcher $d$. Let $\set{e(u)}{u \in S}$ be the canonical basis of $\R^S$. The transitions of $\bX$ can only increase or decrease by one unit the occupancy of one server at a time, so it is convenient to let
\begin{equation*}
\Delta_u^+f(x) \defeq f\left(x + e(u)\right) - f(x) \quad \text{and} \quad \Delta_u^-f(x) \defeq f\left(x - e(u)\right) - f(x) \quad \text{for all} \quad x \in \N^S.
\end{equation*}

It is now possible to check that $Af(x)$ can be expressed as
\begin{equation}
\label{eq: drift with respect to a load balancing process}
Af(x) = \sum_{u \in S}\left[\sum_{d \in \calN(u)} \Delta_u^+f(x)\frac{\lambda(d)}{\left|\calM(d, x)\right|}\ind{u \in \calM(d, x)} + \Delta_u^-f(x)\mu(u)\ind{x(u) > 0}\right].
\end{equation}
The inner summation ranges over all the dispatchers $d$ compatible with $u$. For each $d$, it adds the amount of change in $f$ when $x(u)$ increases by one task times the rate at which $u$ receives a task from $d$. The second term inside the brackets is just the amount of change in $f$ when $x(u)$ decreases by one task times the departure rate of tasks from $u$.

\subsection{Some useful properties}
\label{sub: some useful properties}

An important property of the drift of a function is that its expectation with respect to a stationary distribution is often zero. The following proposition provides a condition for this. It is a special case of a more general result proved in \cite[Proposition 3]{glynn2008bounding}.

\begin{proposition}
	\label{prop: zero-mean drift condition}
	Suppose that $X$ is a stationary distribution of the load balancing process $\bX$ having generator $A$, and $\map{f}{\N^S}{\R}$ is any function. If
	\begin{equation}
	\label{eq: condition in glynn and zeevi}
	E\left|A(X, X) f(X)\right| < \infty, \quad \text{then} \quad E\left[Af(X)\right] = 0.
	\end{equation}
	Moreover, the latter condition holds if $E\left|f(X)\right| < \infty$.
\end{proposition}

\begin{proof}
	Condition \eqref{eq: condition in glynn and zeevi} is proved in \cite[Proposition 3]{glynn2008bounding} for any jump Markov process with a discrete state space, and in particular holds for any load balancing process. For the condition that we state subsequently, observe that
	\begin{equation*}
	\left|A(x, x)\right| \leq \sum_{d \in D} \lambda(d) + \sum_{u \in S} \mu(u) \quad \text{for all} \quad x \in \N^S.
	\end{equation*}
	Hence, $E\left|f(X)\right| < \infty$ implies that $E\left|A(X, X) f(X)\right| < \infty$.
\end{proof}

The latter result can be used to check the condition in the following proposition.

\begin{proposition}
	\label{prop: basic inequality}
	Given $u \in S$, define $\map{f_u}{\N^S}{\R}$ such that $f_u(x) = x(u)$ for all $x \in \N^S$. If $\bX$ has a stationary distribution $X$ and $E\left[Af_u(X)\right] = 0$, then
	\begin{equation*}
	\sum_{d \in \calN(u)} \frac{\lambda(d)}{\left|\calN(d)\right|}P\left(u \in \calM(d, X)\right) \leq \mu(u).
	\end{equation*}
\end{proposition}

\begin{proof}
	It follows from \eqref{eq: drift with respect to a load balancing process} and $E\left[Af_u(X)\right] = 0$ that
	\begin{equation*}
	\sum_{d \in \calN(u)} E\left[\frac{\lambda(d)}{\left|\calM(d, X)\right|}\ind{u \in \calM(d, X)}\right] = \mu(u)P\left(X(u) > 0\right) \leq \mu(u).
	\end{equation*}
	The proof is completed by noting that $\calM(d, x) \subset \calN(d)$ for all $d \in D$ and $x \in \N^S$.
\end{proof}

We conclude with a crucial property that is a particular case of a much more general result; the proof of the general version can be found in \cite[Proposition 4.9.2]{ethier2009markov}.

\begin{proposition}
	\label{prop: result from ethier and kurtz}
	Let $X$ be a random variable with values in $\N^S$ such that $E\left[Af(X)\right] = 0$ for all bounded functions $\map{f}{\N^S}{\R}$. Then $X$ is a stationary distribution for $\bX$.
\end{proposition}

\begin{proof}
	The martingale problem for $A$ is well-posed. Indeed, the existence of solutions follows from \cite[Proposition 4.1.7]{ethier2009markov} and the uniqueness follows from \cite[Theorem 4.4.1]{ethier2009markov}. Further, the set of bounded functions is clearly separating and a core for $A$. Hence, the claim follows directly from \cite[Proposition 4.9.2]{ethier2009markov}.
\end{proof}

\section{Dandelion networks}
\label{sec: dandelion networks}

In this section we analyze the limiting stationary behavior of load balancing processes associated with dandelion networks, which are formally defined below.

\begin{definition}
	\label{def: dandelion network}
	Let $D \neq \emptyset$ be some finite set, and consider finite and disjoint sets $C$ and $\set{B_d}{d \in D}$ such that all the sets $B_d$ have the same cardinality. Define
	\begin{equation*}
	S \defeq \bigcup_{d \in D} B_d \cup C \quad \text{and} \quad E \defeq \bigcup_{d \in D} \{d\} \times \left(B_d \cup C\right).
	\end{equation*}
	The bipartite graph $G = (D, S, E)$ is called a dandelion network.
\end{definition}

A dandelion network is shown in Figures \ref{fig: dandelion network} and \ref{fig: connected components}. The set of central servers $C$ has size $|C| = 4$ and all the sets of boundary servers $B_d$ have size $|B_d| = 3$. Each dispatcher $d$ is compatible with all the central servers and only with the boundary servers in $B_d$. Hence, a task arriving at $d$ is sent to a server with the shortest queue in $B_d \cup C$.

\begin{figure}
	\centering
	\begin{subfigure}{0.49\columnwidth}
		\centering
		\includegraphics{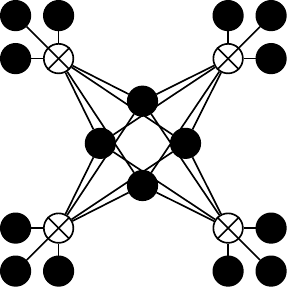}
		\caption{Dandelion network}
	\end{subfigure}%
	\hfill
	\begin{subfigure}{0.49\columnwidth}
		\centering
		\includegraphics{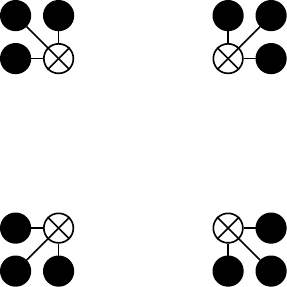}
		\caption{Connected components}
	\end{subfigure}
	\caption{Connected components obtained after removing the central servers of a dandelion network.}
	\label{fig: connected components}
\end{figure}

We consider dandelion networks $G_n = (D_n, S_n, E_n)$ where the number of dispatchers is $n \geq 1$ and the numbers of central servers and boundary servers per dispatcher are fixed. In order to simplify the notation, we assume that the set $C$ of central servers remains fixed, we let $D_n \defeq \{1, \dots, n\}$ be the sets of dispatchers and we assume that the set $B_d$ of boundary servers that are compatible with dispatcher $d$ is fixed for all $n \geq d$.

We let $\bX_n$ be the load balancing process associated with $G_n$ such that all dispatchers have the same arrival rate $\lambda > 0$ and all servers have the same service rate $\mu > 0$. We assume that $\lambda < \mu \left|B_1\right|$, which implies that \eqref{eq: ergodicity condition} holds and hence $\bX_n$ is ergodic. Our focus is on the stationary distributions $X_n$ of the processes $\bX_n$.

\subsection{Main results}
\label{sub: main results}

Any dandelion network is connected, but if the central servers are removed, then we obtain a collection of many connected components as in Figure \ref{fig: connected components}. Each component is as the bipartite graph described in the following definition, and consists of one dispatcher $d$ and the boundary servers $B_d$ that are compatible with the dispatcher.

\begin{definition}
	\label{def: jsq system}
	Let $G = \left(D, S, D \times S\right)$ be a bipartite graph where $D = \{d\}$ is a singleton. A load balancing process is called a basic load balancing process if it is associated with a bipartite graph of the latter form and has constant rate functions. Every incoming task is sent to a server in $S$ with the least number of tasks.
\end{definition}

Our first result for the dandelion networks characterizes the stationary behavior of the boundary servers in the limit as $n \to \infty$. It says that the boundary servers behave as if the central servers were removed from the network; i.e., the boundary servers behave asymptotically as the servers of independent basic load balancing processes. In particular, any performance benefit associated with the central servers disappears in the limit.

\begin{theorem}
	\label{the: stationary asymptotic behavior of boundary servers}
	Let $D$ be a finite set of positive integers and $B_D \defeq \set{b \in B_d}{d \in D}$. Then
	\begin{equation*}
	\vct{X_n(b)}{b \in B_D} \Rightarrow Y \quad \text{as} \quad n \to \infty,
	\end{equation*}
	where $Y$ is the random variable with values in $\N^{B_D}$ such that:
	\begin{enumerate}
		\item[(a)] $\set{\vct{Y(b)}{b \in B_d}}{d \in D}$ are independent and indentically distributed,
		
		\item[(b)] $\vct{Y(b)}{b \in B_d}$ is the stationary distribution of a basic load balancing process where the dispatcher has arrival rate $\lambda$ and the $|B_d|$ servers each have service rate $\mu$.
	\end{enumerate}
\end{theorem}

The proof is carried out in Section \ref{sub: asymptotic behavior of boundary servers} using Proposition \ref{prop: result from ethier and kurtz}. There we consider the generators $A_n$ and $B$ of $X_n$ and $Y$, respectively, as well as the drifts, with respect to the latter generators, of bounded functions that only depend on the occupancies of the servers in $B_D$. We establish that the difference between the expectations with respect to $X_n$ of the drifts with respect to $A_n$ and $B$ vanishes as $n \to \infty$ for any such function. Hence, Proposition \ref{prop: zero-mean drift condition} implies that the mean with respect to $X_n$ of the drift with respect to $B$ of any bounded function approaches zero as $n \to \infty$. If the random variables $X_n$ converge weakly, then it follows that their limit satisfies the conditions of Proposition \ref{prop: result from ethier and kurtz} for $B$. We complete the proof through a tightness result and the latter observation.

The following result is the counterpart of Theorem \ref{the: main result}.

\begin{theorem}
	\label{the: saturation of central servers}
	The central servers saturate in the limit. Namely,
	\begin{equation*}
		\lim_{n \to \infty} P\left(X_n(c) \geq k\right) = 1 \quad \text{for all} \quad k \in \N \quad \text{and} \quad c \in C.
	\end{equation*}
	In particular, $E\left[X_n(c)\right] \to \infty$ as $n \to \infty$ for all $c \in C$.
\end{theorem}

The proof is given in Section \ref{sub: saturation of central servers} and uses Proposition \ref{prop: basic inequality} and Theorem \ref{the: stationary asymptotic behavior of boundary servers}. The latter proposition is used to show that the mean of $M_n \defeq \left|\set{d \in D_n}{C \cap \calM_n(d, X_n) \neq \emptyset}\right|$ is uniformly bounded across $n$. For $n - M_n$ dispatchers there exists a boundary server that has fewer tasks than all the central servers, because $\calM_n(d, X_n)$ is the set of servers that are compatible with dispatcher $d$ and have the minimum number of tasks. Furthermore, the dispatchers, with their compatible boundary servers, are independent and identically distributed in the limit by Theorem \ref{the: stationary asymptotic behavior of boundary servers}. It follows that the minimum number of tasks across the central servers is lower bounded by a diverging number $n - M_n$ of asymptotically independent and identically distributed random variables with unbounded support. The saturation of the central servers is proved combining these observations.


\subsection{Limiting behavior of boundary servers}
\label{sub: asymptotic behavior of boundary servers}

The following technical lemma is proved in Appendix \ref{app: proofs of various results}.

\begin{lemma}
	\label{lem: occupancy upper bound}
	Recall that $B_D \defeq \set{b \in B_d}{d \in D}$. The following properties hold.
	\begin{enumerate}
		\item[(a)] $\set{\vct{X_n(b)}{b \in B_D}}{n \geq \max D}$ is tight for all $D \subset \N$.
		
		\item[(b)] $E\left[X_n(u)\right] < \infty$ for all $u \in S_n$ and $n \geq 1$.
	\end{enumerate}
\end{lemma}

As noted earlier, property (a) will be used in the proof of Theorem \ref{the: stationary asymptotic behavior of boundary servers}. On the other hand, property (b) together with Proposition \ref{prop: zero-mean drift condition} imply that the inequality of Proposition \ref{prop: basic inequality} holds for any server of the dandelion network. This is used to prove the next lemma.

\begin{lemma}
	\label{lem: probability of central server being minimizer}
	For each central server $c \in C$ and each dispatcher $d \geq 1$, we have 
	\begin{equation*}
	\lim_{n \to \infty} P\left(c \in \calM_n\left(d, X_n\right)\right) = 0.
	\end{equation*}
\end{lemma}

\begin{proof}
	Fix $c \in C$ and note that Proposition \ref{prop: basic inequality} yields
	\begin{equation*}
	\frac{\lambda}{|B_1| + |C|}\sum_{d \in D_n} P\left(c \in \calM_n\left(d, X_n\right)\right) = \sum_{d \in D_n} \frac{\lambda}{\left|\calN_n(d)\right|}P\left(c \in \calM_n\left(d, X_n\right)\right) \leq \mu.
	\end{equation*}
	By symmetry, all the terms in the first summation are equal. If we fix $d \geq 1$, then
	\begin{equation}
	\label{eq: bound for probability of central server being minimizer}
	P\left(c \in \calM_n\left(d, X_n\right)\right) \leq \frac{\mu\left(|B_1| + |C|\right)}{\lambda n} \quad \text{for all} \quad n \geq d.
	\end{equation}
	This completes the proof.
\end{proof}

Suppose that $D \subset \N$ and $\map{f}{\N^{B_D}}{\R}$ is a function. For each $n \geq \max D$, we have $B_D \subset S_n$ and thus we may also regard $f$ as a function defined on $\N^{S_n}$ in the natural way. Specifically, the value of $f$ at $x \in \N^{S_n}$ is defined as the value of $\vct{x(b)}{b \in B_D}$. The following lemma concerns the limiting mean with respect to $X_n$ of the drift of such functions with respect to the generator $A_n$ of $\bX_n$.

\begin{lemma}
	\label{lem: asymptotic drift in steady state}
	Fix $D \subset \N$ and a bounded function $\map{f}{\N^{B_D}}{\R}$. We have
	\begin{equation*}
	\lim_{n \to \infty} \left|E\left[A_nf(X_n)\right] - E\left[Bf(X_n)\right]\right| = 0,
	\end{equation*}
	where $B$ is the generator of the process $Y$ in the statement of Theorem \ref{the: stationary asymptotic behavior of boundary servers}. Specifically, let $\calM(d, x) \defeq \mathrm{argmin} \set{x(b)}{b \in B_d}$ for all $x \in \N^{B_D}$ and $d \in D$. Then
	\begin{equation*}
	Bf(x) = \sum_{d \in D} \sum_{b \in B_d} \left[\Delta_b^+f(x) \frac{\lambda}{\left|\calM(d, x)\right|}\ind{b \in \calM(d, x)} + \Delta_b^-f(x)\mu\ind{x(b) > 0}\right] \quad \text{for all} \quad x \in \N^{B_D}.
	\end{equation*}
\end{lemma}

\begin{proof}
	Since $f$ only depends on the occupancies of servers $b \in B_D$, it follows from \eqref{eq: drift with respect to a load balancing process} that
	\begin{equation*}
	A_nf(x) = \sum_{d \in D}\sum_{b \in B_d} \left[\Delta_b^+f(x)\frac{\lambda}{\left|\calM_n(d, x)\right|}\ind{b \in \calM_n(d, x)} + \Delta_b^-f(x)\mu\ind{x(b) > 0}\right]
	\end{equation*}
	for all $x \in \N^{S_n}$. Hence, it suffices to prove that each $d \in D$ and $b \in B_d$ satisfy
	\begin{equation}
	\label{aux: drift difference}
	\lim_{n \to \infty} \left|E\left[\frac{\Delta_b^+f(X_n)}{\left|\calM_n(d, X_n)\right|}\ind{b \in \calM_n(d, X_n)}\right] - E\left[\frac{\Delta_b^+f(X_n)}{\left|\calM(d, X_n)\right|}\ind{b \in \calM(d, X_n)}\right]\right| = 0.
	\end{equation}
	Here $\calM_n(d, X_n)$ is the set of servers that minimize $X_n(u)$ over $u \in C \cup B_d$ and $\calM(d, X_n)$ is the set of servers minimizing $X_n(b)$ over $b \in B_d$.
	
	First observe that
	\begin{align*}
	\frac{\Delta_b^+f(x)}{\left|\calM_n(d, x)\right|}\ind{b \in \calM_n(d, x)} &= \frac{\Delta_b^+f(x)}{\left|\calM_n(d, x)\right|}\ind{b \in \calM_n(d, x), C \cap \calM_n(d, x) \neq \emptyset} \\
	&+ \frac{\Delta_b^+f(x)}{\left|\calM_n(d, x)\right|}\ind{b \in \calM_n(d, x), C \cap \calM_n(d, x) = \emptyset}
	\end{align*}
	for all $d \in D$, $b \in B_d$, $n \geq \max D$ and $x \in \N^{S_n}$. The first term on the right is such that
	\begin{equation*}
	E\left[\frac{\left|\Delta_b^+f(X_n)\right|}{\left|\calM_n(d, X_n)\right|}\ind{b \in \calM_n(d, X_n), C \cap \calM_n(d, X_n) \neq \emptyset}\right] \leq 2\norm{f}_\infty P\left(C \cap \calM_n(d, X_n) \neq \emptyset\right),
	\end{equation*}
	where $\norm{f}_\infty \defeq \sup \set{f(x)}{x \in \N^{S_n}} < \infty$. Moreover,
	\begin{equation*}
	\left|E\left[\frac{\Delta_b^+f(X_n)}{\left|\calM_n(d, X_n)\right|}\ind{b \in \calM_n(d, X_n), C \cap \calM_n(d, X_n) = \emptyset}\right] - E\left[\frac{\Delta_b^+f(X_n)}{\left|\calM(d, X_n)\right|}\ind{b \in \calM(d, X_n)}\right]\right| 
	\end{equation*}
	is also upper bounded by $2\norm{f}_\infty P\left(C \cap \calM_n(d, X_n) \neq \emptyset\right)$. This follows from the fact that
	\begin{equation*}
	\frac{1}{|\calM_n(d, x)|}\ind{b \in \calM_n(d, x), C \cap \calM_n(d, x) = \emptyset} = \frac{1}{|\calM(d, x)|}\ind{b \in \calM(d, x), C \cap \calM_n(d, x) = \emptyset} \quad \text{for all} \quad x \in \N^{S_n},
	\end{equation*}
	which holds because $\calM_n(d, x) = \calM(d, x)$ whenever $C \cap \calM_n(d, x) = \emptyset$.
	
	By Lemma \ref{lem: probability of central server being minimizer}, we have
	\begin{equation*}
	\lim_{n \to \infty} P\left(C \cap \calM_n(d, X_n) \neq \emptyset\right) \leq \lim_{n \to \infty} \sum_{c \in C} P\left(c \in \calM_n(d, X_n)\right) = 0.
	\end{equation*}
	Therefore, we conclude that \eqref{aux: drift difference} holds.
\end{proof}

We are now ready to prove Theorem \ref{the: stationary asymptotic behavior of boundary servers}.

\begin{proof}[Proof of Theorem \ref{the: stationary asymptotic behavior of boundary servers}]
	By Lemma \ref{lem: occupancy upper bound}, every subsequence of $\set{\vct{X_n(b)}{b \in B_D}}{n \geq \max D}$ has a further subsequence that converges weakly to some random variable $X$ with values in $\N^{B_D}$. It suffices to establish that $X$ satisfies (a) and (b) regardless of the subsequence under consideration. For this purpose we may assume without any loss of generality that $\vct{X_n(b)}{b \in B_D} \Rightarrow X$ as $n \to \infty$, instead of fixing a convergent subsequence.
	
	Let $B$ be the generator of the stationary Markov process $Y$ introduced in the statement of the theorem. If $\map{f}{\N^{B_D}}{\R}$ is any bounded function, then
	\begin{equation}
	\label{aux: zero drift}
	E\left[Bf(X)\right] = \lim_{n \to \infty} \left(E\left[Bf(X_n)\right] - E\left[A_nf(X_n)\right]\right) = 0.
	\end{equation}
	Indeed, recall that $Bf$ is bounded if $f$ is; thus, $E\left[Bf(X_n)\right] \to E\left[Bf(X)\right]$ as $n \to \infty$. Also, $E\left[A_nf(X_n)\right] = 0$ by Proposition \ref{prop: zero-mean drift condition}. Then \eqref{aux: zero drift} follows from Lemma \ref{lem: asymptotic drift in steady state}.
	
	Proposition \ref{prop: result from ethier and kurtz} and \eqref{aux: zero drift} imply that $X$ is a stationary distribution of the Markov process with generator $B$. This stationary distribution is unique since the ergodicity condition \eqref{eq: ergodicity condition} holds. Thus, $X$ has the same law as $Y$, and in particular satisfies (a) and (b).
\end{proof}

\subsection{Saturation of central servers}
\label{sub: saturation of central servers}

The proof of Theorem \ref{the: saturation of central servers} relies on the following lemma and Theorem \ref{the: stationary asymptotic behavior of boundary servers}. The lemma implies that the steady-state average number of dispatchers that can send tasks to the central servers is uniformly bounded across $n$. In particular, this means that the average number of boundary servers with more tasks than some central server is uniformly bounded across $n$. The asymptotic independence property established in Theorem \ref{the: stationary asymptotic behavior of boundary servers} will be used to leverage the latter observation, so as to conclude that the central servers saturate.

\begin{lemma}
	\label{lem: number of dispatchers with highly loaded boundary servers}
	If $M_n \defeq \left|\set{d \in D_n}{C \cap \calM_n(d, X_n) \neq \emptyset}\right|$, then
	\begin{equation*}
	\limsup_{n \to \infty} E\left[M_n\right] \leq \frac{\mu\left(|B_1| + |C|\right)|C|}{\lambda}.
	\end{equation*}
\end{lemma}

\begin{proof}
	First observe that
	\begin{align*}
	E\left[M_n\right] &= E\left[\sum_{d \in D_n} \ind{C \cap \calM_n(d, X_n) \neq \emptyset}\right] \\
	&= \sum_{d \in D_n} P\left(C \cap \calM_n(d, X_n) \neq \emptyset\right) \leq \sum_{c \in C} \sum_{d \in D_n} P\left(c \in \calM_n(d, X_n)\right).
	\end{align*}
	Then it follows from \eqref{eq: bound for probability of central server being minimizer} that
	\begin{equation*}
	E\left[M_n\right] \leq \sum_{c \in C} \sum_{d \in D_n} \frac{\mu\left(|B_1| + |C|\right)}{\lambda n} =  \frac{\mu\left(|B_1| + |C|\right)|C|}{\lambda},
	\end{equation*}
	 which completes the proof.
\end{proof}

We are now ready to prove Theorem \ref{the: saturation of central servers}.

\begin{proof}[Proof of Theorem \ref{the: saturation of central servers}]
	Fix $c \in C$ and $k \in \N$. For each $l \in \N$, we have
	\begin{equation}
	\label{aux: inequality 1 theorem 3}
	P\left(X_n(c) < k\right) \leq P\left(M_n \geq l + 1\right) + P\left(X_n(c) < k, M_n \leq l\right). 
	\end{equation}
	By Lemma \ref{lem: number of dispatchers with highly loaded boundary servers}, there exists $\gamma >  \mu\left(|B_1| + |C|\right)|C|/\lambda$ such that
	\begin{equation}
	\label{aux: inequality 2 theorem 3}
	P\left(M_n \geq l + 1\right) \leq \frac{E\left[M_n\right]}{l + 1} \leq \frac{\gamma}{l + 1} \quad \text{for all} \quad n \geq 1 \quad \text{and} \quad l \in \N.
	\end{equation}
	
	It is clear that
	\begin{equation}
	\label{aux: inequality 3 theorem 3}
	P\left(X_n(c) < k, M_n \leq l\right) = \sum_{r = 0}^l P\left(X_n(c) < k, M_n = r\right).
	\end{equation}
	Furthermore, for each $r \in \{0, \dots, l\}$ and $D \subset D_n$, we have
	\begin{equation}
	\label{aux: inequality 4 theorem 3}
	\begin{split}
	P\left(X_n(c) < k, M_n = r\right) &\leq P\left(\max_{d \in D}\min_{b \in B_d} X_n(b) < X_n(c), X_n(c) < k, M_n = r\right) \\
	&+ P\left(\max_{d \in D} \min_{b \in B_d} X_n(b) \geq X_n(c), M_n = r\right) \\
	&\leq P\left(\max_{d \in D}\min_{b \in B_d} X_n(b) < k\right) + 1 - \binom{n - |D|}{r}\binom{n}{r}^{-1}.
	\end{split}
	\end{equation}
	For the inequality in the last line, note that the inequality inside the probability sign of the second line implies that $C \cap \calM_n(d, X_n) \neq \emptyset$ for some $d \in D$. In particular, $d$ is one of the $M_n = r$ dispatchers that satisfy the latter condition. By symmetry, all the dispatchers are equally likely to satisfy the condition, which gives the expression in the last line.
	
	It is possible to check that $(n - |D| - s)n \geq (n - s)(n - 2|D|)$ if $2s \leq n$. Thus,
	\begin{equation}
	\label{aux: inequality 5 theorem 3}
	\binom{n - |D|}{r}\binom{n}{r}^{-1} = \prod_{s = 0}^{r - 1} \frac{n - |D| - s}{n - s} \geq \prod_{s = 0}^{r - 1} \frac{n - 2|D|}{n} = \left(\frac{n - 2|D|}{n}\right)^r
	\end{equation}
	whenever $2(r - 1) \leq n$ and $2|D| \leq n$. Combining \eqref{aux: inequality 1 theorem 3}-\eqref{aux: inequality 5 theorem 3}, we obtain
	\begin{equation*}
	P\left(X_n(c) < k\right) \leq \frac{\gamma}{l + 1} + \left(l + 1\right)\left[P\left(\max_{d \in D}\min_{b \in B_d} X_n(b) < k\right) + 1 - \left(\frac{n - 2|D|}{n}\right)^l\right]
	\end{equation*}
	for all $l \in \N$ such that $2(l - 1) \leq n$, $D \subset D_n$ such that $2|D| \leq n$ and $n \geq 1$.
	
	Let $X$ be a basic load balancing process with arrival rate $\lambda$, service rate $\mu$ and servers indexed by the set $B_1$. Then it follows from Theorem \ref{the: stationary asymptotic behavior of boundary servers} that
	\begin{equation*}
	\limsup_{n \to \infty} P\left(X_n(c) < k\right) \leq \frac{\gamma}{l + 1} + \left(l + 1\right)\left[P\left(\min_{b \in B_1} X(b) < k\right)\right]^{|D|}
	\end{equation*}
	for all $l \in \N$ and $D \subset D_n$. Note that the probability on the right-hand side is independent of $l$ or $D$ and is strictly less than one. Hence, taking $D = \{1, \dots, l\}$ and letting $l \to \infty$ on both sides of the inequality, we may conclude that
	\begin{equation*}
	\lim_{n \to \infty} P\left(X_n(c) < k\right) = 0.
	\end{equation*}
	This completes the proof.
\end{proof}

\section{Monotone transformations}
\label{sec: monotone transformations}

In this section we introduce transformations, of bipartite graphs and rate functions, that have certain monotonicity properties. Informally speaking, these monotonicity properties imply that the load balancing process associated with the transformed bipartite graph and rate functions has fewer tasks at each server in a stochastic dominance sense.

Let $\bX_1$ be a load balancing process associated with a bipartite graph $G_1 = (D_1, S_1, E_1)$ and rate functions $\map{\lambda_1}{D_1}{(0, \infty)}$ and $\map{\mu_1}{S_1}{(0, \infty)}$. One of the transformations that we consider involves coupling the potential departure processes of certain servers, and we need to apply this transformation multiple times to prove Theorem \ref{the: main result}. We thus assume that some servers may have the same potential departure process and associate with $\bX_1$ a partition $\calS_1$ of $S_1$ such that all the servers in $U \in \calS_1$ have the same potential departure process. Namely, the servers in $U$ have potential departures at the jump times of some common Poisson process, and a potential departure from a server leads to an actual departure if the server is not idle. Clearly, $\mu_1(u) = \mu_1(v)$ if $u, v \in U$ and $U \in \calS_1$.

The bipartite graph and rate functions obtained after some given transformation are denoted by $G_2 = (D_2, S_2, E_2)$, $\map{\lambda_2}{D_2}{(0, \infty)}$ and $\map{\mu_2}{S_2}{(0, \infty)}$. The partition of $S_2$ indicating the servers with a common potential departure process is denoted by $\calS_2$ and the associated load balancing process is $\bX_2$. All the transformations satisfy that
\begin{equation*}
D_1 = D_2, \quad S_1 \subset S_2, \quad \lambda_1(d) \geq \lambda_2(d) \quad \text{and} \quad \mu_1(u) \leq \mu_2(u), \quad \text{for} \quad d \in D_1 \quad \text{and} \quad u \in S_1.
\end{equation*}
If $\bX_1(0)$ and $\vct{\bX_2(0, u)}{u \in S_1}$ are identically distributed random variables, then all the monotonicity properties that we prove imply that
\begin{equation*}
P\left(\bX_1(t, u) \geq k\right) \geq P\left(\bX_2(t, u) \geq k\right) \quad \text{for all} \quad t \geq 0, \quad u \in S_1 \quad \text{and} \quad k \in \N.
\end{equation*}

\subsection{Coupled constructions of sample paths}
\label{sub: coupled constructions of sample paths}

The proofs of the monotonicity properties are based on coupled constructions of the sample paths of $\bX_1$ and $\bX_2$. In particular, these constructions are such that
\begin{equation}
\label{eq: inequality between server occupancies}
\bX_1(t, u) \geq \bX_2(t, u) \quad \text{for all} \quad t \geq 0 \quad \text{and} \quad u \in S_1
\end{equation}
holds for each sample path. We prove this by induction, noting that \eqref{eq: inequality between server occupancies} holds at time zero by construction and establishing that \eqref{eq: inequality between server occupancies} is preserved by arrivals and departures.

The following elements are common to all the constructions that we consider.
\begin{itemize}
	\item \emph{Initial conditions.} We postulate that all the servers $u \in S_1$ have the same initial occupancy for $\bX_1$ and $\bX_2$, i.e., $\bX_1(0, u) = \bX_2(0, u)$.
	
	\item \emph{Arrival processes.} The arrival process of dispatcher $d$ for $\bX_2$ is a thinning of the arrival process of the same dispatcher for $\bX_1$.
	
	\item \emph{Potential departure processes.} The potential departure process of server $u \in S_1$ for $\bX_1$ is a thinning of the potential departure process of $u$ for $\bX_2$. 
\end{itemize}
We also need to couple the dispatching decisions so that each arrival preserves \eqref{eq: inequality between server occupancies}. For this purpose we rely on the following lemma, which is proved in Appendix \ref{app: proofs of various results}.

\begin{lemma}
	\label{lem: coupling of dispatching decisions}
	Fix a dispatcher $d \in D_1$ and suppose that there exists an injective function $\map{\varphi_d}{\calN_1(d)}{\calN_2(d)}$. Furthermore, assume that $x_1 \in \N^{S_1}$ and $x_2 \in \N^{S_2}$ are such that
	\begin{equation}
		\label{aux: monotonicity before arrival}
		x_1(u) \geq x_2\left(\varphi_d(u)\right) \quad \text{for all} \quad u \in \calN_1(d).
	\end{equation}
	Then there exists a couple of random variables $U_1$ and $U_2$ that are uniformly distributed over $\calM_1(d, x_1)$ and $\calM_2(d, x_2)$, respectively, and such that
	\begin{equation}
		\label{aux: monotonicity after arrival}
		x_1(u) + \ind{U_1 = u} \geq x_2\left(\varphi_d(u)\right) + \ind{U_2 = \varphi_d(u)} \quad \text{for all} \quad u \in \calN_1(d).
	\end{equation}
\end{lemma}

As noted earlier, the lemma will be used to show that \eqref{eq: inequality between server occupancies} is preserved by arrivals; the definition of $\varphi_d$ will depend on the transformation under consideration, but the general idea is the same in all cases. More specifically, suppose that $\bX_1$ and $\bX_2$ satisfy \eqref{aux: monotonicity before arrival} right before an arrival at dispatcher $d$ for $\bX_2$; this is a more general version of \eqref{eq: inequality between server occupancies}, saying that server $u$ has more tasks for $\bX_1$ than $\varphi_d(u)$ has for $\bX_2$. The arrival process associated with $d$ for $\bX_2$ will be a thinning of the arrival process associated with $d$ for $\bX_1$, and thus a task arrives at $d$ also for $\bX_1$. Lemma \ref{lem: coupling of dispatching decisions} shows that there exist coupled random variables $U_1$ and $U_2$ that capture the dispatching decisions taken at $d$ for each process, with the appropriate marginal laws, and preserve \eqref{aux: monotonicity before arrival} after the task has been dispatched in both systems.

\subsection{Transformations and monotonicity properties}
\label{sub: transformations and monotonicity properties}

The first transformation is defined below and depicted in Figure \ref{fig: edge simplification}.

\begin{figure}
	\centering
	\begin{subfigure}{0.49\columnwidth}
		\centering
		\includegraphics{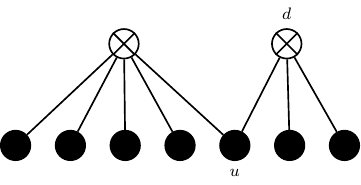}
		\caption{Before}
	\end{subfigure}%
	\hfill
	\begin{subfigure}{0.49\columnwidth}
		\centering
		\includegraphics{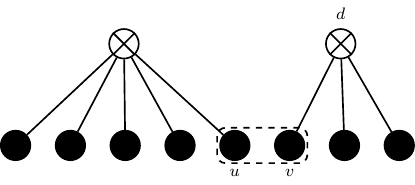}
		\caption{After}
	\end{subfigure}
	\caption{Edge simplification that removes the compatibility relation $(d, u)$ and incorporates server $v$ and the compatibility relation $(d, v)$. The servers $u$ and $v$ have the same potential departure process.}
	\label{fig: edge simplification}
\end{figure}

\begin{definition}
	\label{def: edge simplification}
	An edge simplification removes a compatibility relation $(d, u) \in E_1$ and incorporates a server $v \notin S_1$ and the compatibility relation $(d, v)$. The bipartite graph that results from this transformation is such that
	\begin{equation*}
	D_2 \defeq D_1, \quad S_2 \defeq S_1 \cup \{v\} \quad \text{and} \quad E_2 \defeq \left(E_1 \setminus \{(d, u)\}\right) \cup \{(d, v)\}.
	\end{equation*}
	The edge simplification gives $v$ the potential departure process of $u$. Namely, let $U$ be the element of the partition $\calS_1$ such that $u \in U$. Then
	\begin{equation*}
	\calS_2 \defeq \left(\calS_1 \setminus \{U\}\right) \cup \left\{U \cup \{v\}\right\} \quad \text{and} \quad \mu_2(v) \defeq \mu_1(u).
	\end{equation*}
	In addition, $\lambda_2(e) \defeq \lambda_1(e)$ for all $e \in D_2 = D_1$ and $\mu_2(w) \defeq \mu_1(w)$ for all $w \in S_1$.
\end{definition}

Edge simplification has the following monotonicity property.

\begin{proposition}
	\label{prop: monotonicity of edge simplification}
	Suppose that $\bX_2$ is obtained from $\bX_1$ by means of the edge simplification that removes the compatibility relation $(d, u)$ and incorporates the server $v$. Assume also that $\bX_1(0)$ and $\vct{\bX_2(0, w)}{w \in S_1}$ are identically distributed and $\bX_2(0, u) \geq \bX_2(0, v)$ with probability one. Then the following inequalities hold
	\begin{equation*}
	P\left(\bX_1(t, u) \geq k\right) \geq P\left(\bX_2(t, v) \geq k\right) \quad \text{and} \quad P\left(\bX_1(t, w) \geq k\right) \geq P\left(\bX_2(t, w) \geq k\right)
	\end{equation*}
	for all $t \geq 0$, $w \in S_1$ and $k \in \N$.
\end{proposition}

\begin{proof}
	We construct $\bX_1$ and $\bX_2$ from the following independent stochastic primitives.
	\begin{itemize}
		\item \emph{Initial conditions.} Random variables with the laws of $\bX_1(0)$ and $\bX_2(0, v)$ given the value of $\vct{\bX_2(0, w)}{w \in S_1}$, for each possible value.
		
		\item \emph{Arrival processes.} independent Poisson arrival processes indexed by $e \in D_1$ and such that the process associated with $e$ has intensity $\lambda_1(e)$.
		
		\item \emph{Potential departure processes.} independent Poisson processes indexed by $W \in \calS_1$ and such that the process with index $W$ has intensity $\mu_1(w)$ for all $w \in W$.
		
		\item \emph{Selection variables.} An infinite sequence of indepenent random variables that are uniformly distributed over the interval $[0, 1]$.
	\end{itemize}
	We use the initial conditions to define $\bX_1(0)$ and $\bX_2(0)$ so that $\bX_1(0, w) = \bX_2(0, w)$ for all $w \in S_1$ and $\bX_2(0, u) \geq \bX_2(0, v)$ with probability one. Each dispatcher has the same arrival process for both load balancing processes. For each $W \in \calS_1$, all the servers in $W$ have the same potential departure process for both load balancing processes, and $v$ has the same potential departures process as $u$ for $\bX_2$. The selection variables are used in combination with Lemma \ref{lem: coupling of dispatching decisions} to define the dispatching decisions, as indicated below.
	
	The above construction is such that
	\begin{equation}
	\label{aux: monotonicity property per sample path}
	\bX_1(t, u) \geq \bX_2(t, v) \quad \text{and} \quad \bX_1(t, w) \geq \bX_2(t, w) \quad \text{for all} \quad w \in S_1 \quad \text{and} \quad t = 0.
	\end{equation}
	Moreover, it is clear that any potential departure preserves these inequalities, and we can couple the dispatching decisions such that any arrival also preserves the above inequalities. Specifically, consider the injections $\map{\varphi_e}{\calN_1(e)}{\calN_2(e)}$ defined by
	\begin{equation*}
	\varphi_d(u) \defeq v, \quad \varphi_d(w) \defeq w \quad \text{for} \quad w \in \calN_1(d) \setminus \{u\} \quad \text{and} \quad \varphi_e(w) \defeq w \quad \text{for} \quad w \in \calN_1(e). 
	\end{equation*}
	Suppose that a task arrives at dispatcher $e$ at time $t$, which occurs simultaneously for both load balancing processes. If \eqref{aux: monotonicity property per sample path} holds right before the arrival, then $e$, $\bX_1\left(t^-\right)$ and $\bX_2\left(t^-\right)$ satisfy \eqref{aux: monotonicity before arrival}. It follows that there exists a distribution $(U_1, U_2)$ as in Lemma \ref{lem: coupling of dispatching decisions}. We use two selection variables and inverse transform sampling to construct $(U_1, U_2)$, and we assign the task to servers $U_1$ for $\bX_1$ and $U_2$ for $\bX_2$. Then \eqref{aux: monotonicity property per sample path} holds at time $t$ by Lemma \ref{lem: coupling of dispatching decisions}.
	
	We have established that the coupled construction is such that \eqref{aux: monotonicity property per sample path} holds at time zero and is preserved by arrivals and potential departures. Because $\bX_1$ and $\bX_2$ are constant between the latter events, it follows by induction that \eqref{aux: monotonicity property per sample path} holds for all $t \geq 0$.
\end{proof}

We now define a few more transformations.

\begin{definition}
	\label{def: additional transformations}
	We introduce the following transformations.
	\begin{itemize}
		\item \emph{Arrival rate decrease.} The arrival rate of tasks is decreased for some dispatchers. Specifically, $\lambda_1(d) \geq \lambda_2(d)$ for all $d \in D_1$ while $G_2 \defeq G_1$, $\calS_2 \defeq \calS_1$ and $\mu_2 \defeq \mu_1$.
		
		\item \emph{Service rate increase.} The service rate of tasks is increased for some servers. Namely, $\mu_1(u) \leq \mu_2(u)$ for all $u \in S_1$ while $G_2 \defeq G_1$, $\calS_2 \defeq \calS_1$ and $\lambda_2 \defeq \lambda_1$.
		
		\item \emph{Server addition.} A server $u \notin S_1$ is attached to a dispatcher $d \in D_1$ while $D_2 \defeq D_1$, $\lambda_2 \defeq \lambda_1$ and $\mu_2(v) \defeq \mu_1(v)$ for all $v \in S_1$. In particular,
		\begin{equation*}
		S_2 \defeq S_1 \cup \{u\}, \quad E_2 \defeq E_1 \cup \{(d, u)\}, \quad \calS_2 \defeq \calS_1 \cup \left\{\{u\}\right\} \quad \text{and} \quad \mu_2(u) > 0.
		\end{equation*}
	\end{itemize}
\end{definition}

The latter transformations have the following monotonicity property.

\begin{proposition}
	\label{prop: monotonicity of simple transformations}
	Suppose that $\bX_2$ is obtained from $\bX_1$ through one of the transformations described in Definition \ref{def: additional transformations}. Assume also that $\bX_1(0)$ and $\vct{\bX_2(0, u)}{u \in S_1}$ are indentically distributed. Then the following inequalities hold
	\begin{equation*}
	P\left(\bX_1(t, u) \geq k\right) \geq P\left(\bX_2(t, u) \geq k\right) \quad \text{for all} \quad t \geq 0, \quad u \in S_1 \quad \text{and} \quad k \in \N.
	\end{equation*}
\end{proposition}

\begin{proof}
	As for Proposition \ref{prop: monotonicity of edge simplification}, the proof is based on a coupled construction of $\bX_1$ and $\bX_2$. This construction relies on the following independent stochastic primitives.
	\begin{itemize}
		\item \emph{Initial conditions.} A random variable with the law of $\bX_1(0)$, and if the transformation is a server addition that adds server $v$, also random variables with the law of $\bX_2(0, v)$ given $\vct{\bX_2(0, u)}{u \in S_1}$, for each possible value of the latter vector.
		
		\item \emph{Arrival processes.} For each $d \in D_1$, two Poisson arrival processes such that the arrival processes associated with dispatcher $d$ have intensities $\lambda_1(d)$ for $\bX_1$ and $\lambda_2(d)$ for $\bX_2$, and the latter is a thinning of the former.
		
		\item \emph{Potential departure processes.} For each set $U \in \calS_1$, two Poisson potential departure processes such that for all $u \in U$ the rates of these processes are $\mu_1(u)$ and $\mu_2(u)$ for $\bX_1$ and $\bX_2$, respectively. Further, the former process is a thinning of the latter. An additional Poisson process of rate $\mu_2(v)$ is used if we add a server $v$.
		
		\item \emph{Selection variables.} An infinite sequence of independent random variables that are uniformly distributed over the interval $[0, 1]$.
	\end{itemize}
	
	We use the initial conditions to define $\bX_1(0)$ and $\bX_2(0)$ so that $\bX_1(0, u) = \bX_2(0, u)$ for all $u \in S_1$ , and the selection variables are used in combination with Lemma \ref{lem: coupling of dispatching decisions} to define the dispatching decisions. For this purpose, we define $\map{\varphi_d}{\calN_1(d)}{\calN_2(d)}$ by
	\begin{equation*}
	\varphi_d(u) = u \quad \text{for all} \quad u \in \calN_1(d) \quad \text{and} \quad d \in D_1.
	\end{equation*}

	The rest of the proof proceeds as in Proposition \ref{prop: monotonicity of edge simplification}, noting that every task arrival for $\bX_2$ coincides with a task arrival for $\bX_1$ at the same dispatcher, and every potential departure for $\bX_1$ coincides with a potential departure for $\bX_2$ at the same server.
\end{proof}

The following corollary is a straigthforward consequence of Propositions \ref{prop: monotonicity of edge simplification} and \ref{prop: monotonicity of simple transformations}.

\begin{corollary}
	\label{cor: composition of transformations}
	Suppose that $\bX_2$ is obtained from $\bX_1$ by applying a finite number of the transformations described in Definitions \ref{def: edge simplification} and \ref{def: additional transformations} in a sequential manner. Assume also that $\bX_1(0)$ and $\vct{\bX_2(0, u)}{u \in S_1}$ are identically distributed and $\bX_2(0, u) = 0$ for all $u \notin S_1$ with probability one. Then the following inequalities hold
	\begin{equation*}
	P\left(\bX_1(t, u) \geq k\right) \geq P\left(\bX_2(t, u) \geq k\right) \quad \text{for all} \quad t \geq 0, \quad u \in S_1 \quad \text{and} \quad k \in \N.
	\end{equation*}
\end{corollary}

\section{Proof of the main result}
\label{sec: proof of the main result}

Essentially, the proof of Theorem \ref{the: main result} is based on Theorem \ref{the: saturation of central servers} and Corollary \ref{cor: composition of transformations}. Consider servers $\set{u_n \in S_n}{n \geq 1}$ with skewed neighborhoods. Loosely speaking, we will combine the monotone transformations to obtain networks with an isolated dandelion subnetwork, such that the number of tasks at $u_n$ is stochastically lower bounded by the number of tasks at a central server of the dandelion subnetwork. Further, the size of the latter subnetwork will approach infinity, and thus the saturation of $u_n$ will follow from Theorem \ref{the: saturation of central servers}.

The key transformation is edge simplification, since it allows to prune some of the edges around the skewed neighborhoods. The fact that this transformation synchronizes the potential departures from certain servers complicates the proof, because Theorem \ref{the: saturation of central servers} concerns dandelion networks where the processing times are independent across the servers. Hence, the transformations mentioned above must be combined carefully. Specifically, the potential departures from the servers in the dandelion subnetwork must be independent, but may nonetheless be synchronized with the potential departures from servers in other components of the transformed network; this does not hinder the outlined proof plan.

Formally speaking, we will first prove a weak version of Theorem \ref{the: main result}, which requires a technical condition that makes it easier to obtain dandelion subnetworks where all the servers have independent potential departures. Then we will establish that this condition is in fact automatically satisfied if the assumptions of Theorem \ref{the: main result} hold. In order to state the condition, let $\set{\bX_n}{n \geq 1}$ be load balancing processes as in Section \ref{sec: main result} and define
\begin{equation*}
\calN_n^\alpha(U) \defeq \bigcap_{u \in U} \calN_n^\alpha(u) \quad \text{for all} \quad \alpha \in \N \times (0, \infty) \times (0, \infty) \quad \text{and} \quad U \subset S_n.
\end{equation*}

\begin{condition}
	\label{con: technical condition}
	There exist a tuple $\alpha = (a, \lambda_{\min}, \mu_{\max}) \in \N \times (0, \infty) \times (0, \infty)$ and sequences $\set{U_n \subset S_n, \calA_n \subset \calN_n^\alpha\left(U_n\right)}{n \geq 1}$ such that the following properties hold.
	\begin{enumerate}
		\item[(a)] $|U_n| = c$ for all sufficiently large $n$ and some $1 \leq c \leq a$,
		
		\item [(b)] $\calN_n(d) \cap \calN_n(e) \subset U_n$ for all $d, e \in \calA_n$ and all $n \geq 1$,
		
		\item [(c)] $\left|\calA_n\right| \to \infty$ as $n \to \infty$.
	\end{enumerate}
\end{condition}

\begin{remark}
	\label{rem: consistency of technical condition}
	The constant $c$ introduced in (a) can only take values within $\{1, \dots, a\}$ for the condition to be consistent. Indeed, note that $\calN_n^\alpha(U_n) = \emptyset$ if $|U_n| > a$.
\end{remark}

The following lemma is the weak version of Theorem \ref{the: main result} mentioned above.

\begin{lemma}
	\label{lem: weak version of main result}
	Suppose that Condition \ref{con: technical condition} holds. If $u_n \in U_n$ for all $n \geq 1$, then
	\begin{equation*}
	\lim_{n \to \infty} \liminf_{t \to \infty} P\left(\bX_n\left(t, u_n\right) \geq k\right) = 1 \quad \text{for all} \quad k \in \N.
	\end{equation*}
\end{lemma}

\begin{proof}
	We may assume without any loss of generality that $|U_n| = c$ for all $n \geq 1$. Moreover, the proof is straightforward if $c = a$ since the diverging number of dispatchers in $\calN_n^\alpha(U_n)$ are only compatible with the finite set of servers in $U_n$. Thus, we assume that $c < a$.
	
	We will define load balancing processes $\set{\bY_n}{n \geq 1}$ given by dandelion networks with $c$ central servers and $a - c$ boundary servers for each dispatcher, in such a way that the occupancies of the servers in $U_n$ will be lower bounded by the occupancies of the central servers in a stochastic dominance sense. Then the claim will follow from Theorem \ref{the: saturation of central servers}.
	
	Each dandelion network is obtained by applying finitely many of the transformations defined in Section \ref{sec: monotone transformations}, as indicated below. The superscript $(i)$ refers to the bipartite graph obtained after the transformations applied in step $i$, and the steps are as follows.
	\begin{enumerate}
		\item Choose $0 < \lambda < \lambda_{\min}$ and $\mu > \mu_{\max}$ such that $\lambda < (a - c) \mu$. We decrease the arrival rates of the dispatchers in $\calA_n$ and increase the service rates of the servers that are compatible with at least one of these dispatchers, so that all the latter dispatchers have arrival rate $\lambda$ and all the latter servers have service rate $\mu$.
		
		\item For each $u \in S_n^{(1)}$ such that $u \in \calN_n^{(1)}(d)$ for some $d \in \calA_n$, we perform one edge simplification for each compatibility relation $(e, u)$ such that $e \notin \calA_n$.
		
		\item For each $d \in \calA_n$ such that $\deg_n^{(2)}(d) < a$, we apply server addition transformations, adding servers with service rate $\mu$ until $d$ has exactly $a$ compatible servers.
	\end{enumerate}

	After the second step, all the dispatchers in $\calA_n$ are in a connected component of $G_n^{(2)}$ that does not contain any other dispatcher and contains $U_n$. Further, (b) of Condition \ref{con: technical condition} implies that each server $u \notin U_n$ that lies in this connected component is compatible with exactly one dispatcher. Note that the servers incorporated through the edge simplifications are not in the same connected component as $\calA_n$, and thus all the servers in the same connected component as $\calA_n$ have independent potential departure processes. 
	
	The third step adds servers to the latter connected component until each dispatcher has exactly $a$ compatible servers in total. Therefore, the above observations imply that the connected component of $G_n^{(3)}$ that contains $U_n$ and $\calA_n$ is a dandelion network with set of central servers $U_n$ and set of dispatchers $\calA_n$. Furthermore, the number of central servers is $c$ and each dispatcher is compatible with exactly $a - c$ boundary servers.
	
	Let $\bY_n$ be the load balancing process associated with the latter dandelion network and the rate functions such that all the dispatchers have arrival rate $\lambda$ and all the servers have service rate $\mu$. The choice of $\lambda$ and $\mu$ implies that $\bY_n$ is ergodic. Let $Y_n$ denote the stationary distribution of $\bY_n$ and let us identify the central servers with the servers in $U_n$ in some arbitrary way. It follows from Corollary \ref{cor: composition of transformations} that
	\begin{equation*}
	P\left(\bX_n(t, u) \geq k\right) \geq P\left(\bY_n(t, u) \geq k\right) \quad \text{for all} \quad u \in U_n \quad \text{and} \quad k \in \N
	\end{equation*}
	if $\bY_n(0)$ is the identically zero vector. Moreover, taking the limit inferior as $t \to \infty$ first, and then the limit as $n \to \infty$, we conclude from Theorem \ref{the: saturation of central servers} that
	\begin{equation*}
	\lim_{n \to \infty} \liminf_{t \to \infty} P\left(\bX_n(t, u_n) \geq k\right) \geq \lim_{n \to \infty} P\left(Y_n(u_n) \geq k\right) = 1 \quad \text{for all} \quad k \in \N,
	\end{equation*}
	where $u_n \in U_n$ for all $n \geq 1$, as in the statement of the lemma.
\end{proof}

The following lemma gives a condition that implies Condition \ref{con: technical condition}.

\begin{lemma}
	\label{lem: sufficient technical condition}
	Suppose that there exist $\alpha = (a, \lambda_{\min}, \mu_{\max}) \in \N \times (0, \infty) \times (0, \infty)$ and sequences $\set{U_n \subset S_n, \calB_n \subset \calN_n^\alpha(U_n)}{n \geq 1}$ such that $|U_n| = c \in \{1, \dots, a\}$ for all large enough $n$. Assume in addition that
	\begin{equation}
	\label{aux: condition with bn}
	\lim_{n \to \infty} \frac{\left|\calB_n\right|}{b_n} = \infty \quad \text{with} \quad b_n \defeq \max_{u \notin U_n} \left|\set{d \in \calB_n}{(d, u) \in E_n}\right|.
	\end{equation}
	Then there exist sets $\set{\calA_n \subset \calN_n^\alpha(U_n)}{n \geq 1}$ such that Condition \ref{con: technical condition} holds.
\end{lemma}

\begin{proof}
	Consider the following coloring algorithm.
	\begin{enumerate}
		\item Select an arbitrary uncolored dispatcher $d \in \calB_n$ and color it green.
		
		\item Color red all the dispatchers $e \in \calB_n$ such that $\calN_n(d) \cap \calN_n(e) \not\subset U_n$.
		
		\item Repeat until all the dispatchers in $\calB_n$ have been colored green or red.
	\end{enumerate}
	If we define $\calA_n$ as the set of dispatchers in $\calB_n$ that are green, then it is clear that (b) of Condition \ref{con: technical condition} holds. Moreover, each iteration of the algorithm generates exactly one green dispatcher and at most $\left(a - \left|U_n\right|\right)\left(b_n - 1\right) \leq ab_n$ red dispatchers. We thus conclude that $\left|\calB_n\right| \leq \left|\calA_n\right|\left(1 + ab_n\right)$, and this implies that condition (c) holds as well by \eqref{aux: condition with bn}.
\end{proof}

We are now ready to prove Theorem \ref{the: main result}.

\begin{proof}[Proof of Theorem \ref{the: main result}]
	For each fixed $k \in \N$, it suffices to show that any sequence of natural numbers has a subsequence $\calM$ such that
	\begin{equation}
	\label{aux: claim for subsequence}
	\lim_{m \to \infty} \liminf_{t \to \infty} P\left(\bX_m\left(t, u_m\right) \geq k\right) = 1 \quad \text{with}\ m\ \text{ranging over}\ \calM. 
	\end{equation}
	 	
	We now fix $k \in \N$ and an increasing sequence of natural numbers, which we index by $n$ so as to not introduce additional notation. Next we construct a subsequence $\calM$ and sets $\set{U_m \subset S_m, \calA_m \subset \calN_m^\alpha(U_m)}{m \in \calM}$ such that $u_m \in U_m$ for all $m \in \calM$. We prove \eqref{aux: claim for subsequence} by establishing that Condition \ref{con: technical condition} holds and invoking Lemma \ref{lem: weak version of main result}.
	
	Let $u_n^1 \defeq u_n$ for all $n$ and recall that $\left|\calN_n^\alpha\left(u_n^1\right)\right| \to \infty$ as $n \to \infty$. We now define sequences of servers $u_n^i \in S_n$ in a recursive manner, as follows. Suppose that the $i$th sequence has already been defined. If $S_n \setminus \left\{u_n^1, \dots, u_n^i\right\} \neq \emptyset$, then let
	\begin{equation*}
	u_n^{i + 1} \in \argmax{u \neq u_n^1, \dots, u_n^i} \left|\calN_n^\alpha\left(u_n^1, \dots, u_n^i, u\right)\right|.
	\end{equation*}
	If $\left|\calN_n^\alpha\left(u_n^1, \dots, u_n^{i + 1}\right)\right| \to \infty$ as $n \to \infty$, then we proceed. Otherwise, we stop and let $c \defeq i$. As in Remark \ref{rem: consistency of technical condition}, we may conclude that $1 \leq c \leq a$. Moreover, by construction:
	\begin{equation*}
	\left|\calN_n^\alpha\left(u_n^1, \dots, u_n^c\right)\right| \to \infty \quad \text{and} \quad \max_{u \neq u_n^1, \dots, u_n^c} \left|\calN_n^\alpha\left(u_n^1, \dots, u_n^c, u\right)\right| \nrightarrow \infty \quad \text{as} \quad n \to \infty.
	\end{equation*}
	Therefore, there exists a subsequence $\calM$ such that
	\begin{equation}
	\label{aux: property for constructing an}
	\lim_{m \to \infty} \left|\calN_m^\alpha\left(u_m^1, \dots, u_m^c\right)\right| = \infty \quad \text{and} \quad \limsup_{m \to \infty} \max_{u \neq u_m^1, \dots, u_m^c} \left|\calN_m^\alpha\left(u_m^1, \dots, u_m^c, u\right)\right| < \infty.
	\end{equation}
	
	We define $U_m \defeq \left\{u_m^1, \dots, u_m^c\right\}$ and $\calB_m \defeq \calN_m^\alpha(U_m)$ for all $m \in \calM$, and we note that
	\begin{equation*}
	b_m \defeq \max_{u \notin U_m} \left|\set{d \in \calB_m}{(d, u) \in E_m}\right| = \max_{u \notin U_m} \left|\calN_m^a\left(U_m \cup \{u\}\right)\right|.
	\end{equation*}
	It follows from \eqref{aux: property for constructing an} that the sets $U_m$ and $\calB_m$ satisfy the assumptions of Lemma \ref{lem: sufficient technical condition}. Thus, there exist sets $\set{\calA_m \subset \calN_m^\alpha(U_m)}{m \in \calM}$ such that Condition \ref{con: technical condition} holds.
\end{proof}

\section{Random networks}
\label{sec: random networks}

In this section we provide two examples of randomly generated sequences of networks with skewed neighborhoods. More precisely, we construct the sequences by sampling each network independently and from the same random graph model, such that the size of the network approaches infinity. We establish that there almost surely exists a subsequence of networks with skewed neighborhoods where saturation occurs.

The networks in the first example are given by random bipartite graphs with constant average degree, whereas the networks in the second example are defined using Erd\H{o}s-R\'enyi random graphs with diverging average degree. In the latter case, a result obtained in \cite{mukherjee2018asymptotically} implies that only a vanishing fraction of the servers have more than one task in steady state. However, we prove that server saturation may simultaneously occur.

\subsection{Random bipartite graphs}
\label{sub: random bipartite graphs}

Let us fix a constant $b > 0$. For each $n \geq b$, we let $p_n \defeq b / n \in (0, 1]$ and we define $G_n = (D_n, S_n, E_n)$ as the random bipartite graph such that
\begin{equation*}
D_n \defeq \{1, \dots, n\}, \quad S_n \defeq \{1, \dots, n\} \quad \text{and} \quad \set{\ind{(d, u) \in E_n}}{d \in D_n, u \in S_n}
\end{equation*}
are independent random variables with mean $p_n$, i.e., the edges are drawn independently with probability $p_n$. In particular, the mean degree of both dispatchers and servers is $b$. We assume that all the dispatchers have the same arrival rate $\lambda > 0$ and all the servers have the same service rate $\mu > \lambda$. We prove the following theorem.

\begin{theorem}
	\label{the: random bipartite graphs}
	Let the random bipartite graphs $\set{G_n}{n \geq b}$ be independent and defined on a common probability space. For each $n \geq b$, fix a server $u_n \in S_n$ with maximum degree. Also, fix $a \geq 1$ and let $\alpha \defeq (a, \lambda, \mu)$. With probability one, there exists an infinite subsequence of bipartite graphs, indexed by some set $\calM \subset \N$, such that the servers $\set{u_m}{m \in \calM}$ have a skewed neighborhood. More specifically,
	\begin{equation*}
	\lim_{m \to \infty} \left|\calN_m^\alpha\left(u_m\right)\right| = \infty.
	\end{equation*}
\end{theorem}

Note that the sets $\calN_n^\alpha(u)$ only depend on the structure of $G_n$. Indeed,
\begin{equation*}
\calN_n^\alpha\left(u\right) = \set{d \in \calN_n\left(u\right)}{\deg_n(d) \leq a} \quad \text{for all} \quad u \in S_n,
\end{equation*}
since all the dispatchers have the same arrival rate and all the servers have the same service rate. Moreover, once the graphs $G_n$ have been sampled, each graph defines a load balancing process $\bX_n$. Theorem \ref{the: main result} implies that \eqref{eq: weak limit is infinity} holds for $\set{\bX_m}{m \in \calM}$, which means that the servers with maximum degree $\set{u_m}{m \in \calM}$ saturate.

\begin{remark}
	\label{rem: stability of random bipartite graphs}
	Given $n \geq b$, the load balancing process associated with the network described above is not ergodic with positive probability. Indeed, with positive probability some server is compatible with more than $\mu / \lambda$ dispatchers that are not compatible with any other server; thus, the arrival rate of tasks to the server is larger than its service rate. However, it is possible to modify $G_n$ to obtain a network $G_n' = (D_n', S_n', E_n')$ such that Theorem \ref{the: random bipartite graphs} holds and the associated load balancing process is always ergodic. Let
	\begin{equation*}
	D_n' \defeq D_n, \quad S_n' \defeq S_n \cup \{n + 1, \dots, 2n\} \quad \text{and} \quad E_n' \defeq E_n \cup \set{(d, n + d)}{d \in D_n}.
	\end{equation*}
	In other words, the network $G_n'$ is obtained by attaching a dedicated server $n + d$ to each dispatcher $d$ in $G_n$. The service rate $\mu$ of each dedicated server is larger than the arrival rate $\lambda$ of the corresponding dispatcher. Thus, \eqref{eq: ergodicity condition} holds and the load balancing process is ergodic. Also, it is immediate that $\deg_n'(d) = \deg_n(d) + 1$ for all $d \in D_n$, and hence
	\begin{equation*}
	\calN_n^{\alpha'}(u_n) = \calN_n^\alpha(u_n) \quad \text{with} \quad \alpha' \defeq (a + 1, \lambda, \mu).
	\end{equation*}
	It follows that Theorem \ref{the: random bipartite graphs} holds for the bipartite graphs $G_n'$. While these modified networks always yield an ergodic load balancing process, the saturation property persists.
\end{remark}

\begin{proof}[Proof of Theorem \ref{the: random bipartite graphs}]
	We define
	\begin{equation*}
		q_n^a \defeq \cprob*{d \in \calN_n^\alpha(u) | d \in \calN_n(u)} = \cprob*{\deg_n(d) \leq a | d \in \calN_n(u)};
	\end{equation*}
	since both dispatchers and servers are exchangeable, this quantity does not depend on the specific choice of $d \in D_n$ and $u \in S_n$. Also, $\deg_n(d) - 1$ given that $d \in \calN_n(u)$ is binomially distributed: it is the sum of $n - 1$ independent Bernoulli random variables with mean $p_n$. Thus, it follows from the Poisson limit theorem that
	\begin{equation*}
		\lim_{n \to \infty} q_n^a = \lim_{n \to \infty} \cprob*{\deg_n(d) - 1 \leq a - 1 | d \in \calN_n(u)} = \sum_{k = 0}^{a - 1} \frac{b^k\e^{-b}}{k!}.
	\end{equation*}
	
	For each $u \in S_n$, each dispatcher $d \in \calN_n(u)$ is in $\calN_n^\alpha(u)$ with probability $q_n^a$ and independently from the other dispatchers. As a result,
	\begin{equation*}
	\expect*{\left|\calN_n^\alpha(u)\right| | \deg_n(u)} = \deg_n(u) q_n^a \quad \text{and} \quad \var*{\left|\calN_n^\alpha(u)\right| | \deg_n(u)} = \deg_n(u) q_n^a \left(1 - q_n^a\right).
	\end{equation*}
	It follows from Chebyshev's inequality that
	\begin{align*}
	\cprob*{\left|\calN_n^\alpha(u)\right| \leq \frac{\deg_n(u) q_n^a}{2} | \deg_n(u)} &\leq \cprob*{\left|\left|\calN_n^\alpha(u)\right| - \deg_n(u)q_n^a\right| \geq \frac{\deg_n(u)q_n^a}{2} | \deg_n(u)} \\
	&\leq \frac{4\deg_n(u)q_n^a\left(1 - q_n^a\right)}{\left[\deg_n(u)q_n^a\right]^2} = \frac{4\left(1 - q_n^a\right)}{\deg_n(u)q_n^a}.
	\end{align*}
	
	The degree of any $u \in S_n$ is binomially distributed. By Lemma \ref{lem: lower bound for binomial} of Appendix \ref{app: lemmas used in the examples}, 
	\begin{equation*}
		P\left(\deg_n(u) \geq b + l_n\right) \geq \frac{1}{\log(n)\sqrt{8\left(b + l_n\right)}} \quad \text{with} \quad l_n \defeq \sqrt{\frac{b\log\log(n)}{2}} = \sqrt{\frac{n p_n\log\log(n)}{2}}.
	\end{equation*}
	Recall that $u_n$ has maximum degree, and note that $1 - x \leq \e^{-x}$ for all $x \in \R$. Hence,
	\begin{align*}
		P\left(\deg_n(u_n) < b + l_n\right) &= P\left(\deg_n(u) < b + l_n\ \text{for all}\ u \in S_n\right) \\
		&\leq \left[1 - \frac{1}{\log(n)\sqrt{8\left(b + l_n\right)}}\right]^n \leq \e^{-\frac{n}{\log(n)\sqrt{8\left(b + l_n\right)}}}.
	\end{align*}
	
	Consider the event defined as
	\begin{equation*}
	A_n \defeq \left\{\left|\calN_n^\alpha\left(u_n\right)\right| > \frac{b + l_n}{2}, \deg_n\left(u_n\right) \geq b + l_n\right\}.
	\end{equation*}
	The probability of this event can be lower bounded as follows:
	\begin{align*}
	P\left(A_n\right) &= \cprob*{\left|\calN_n^\alpha\left(u_n\right)\right| > \frac{b + l_n}{2} | \deg_n\left(u_n\right) \geq b + l_n}P\left(\deg_n\left(u_n\right) \geq b + l_n\right) \\
	&\geq \cprob*{\left|\calN_n^\alpha\left(u_n\right)\right| > \frac{\deg_n(u_n)}{2} | \deg_n\left(u_n\right) \geq b + l_n}P\left(\deg_n\left(u_n\right) \geq b + l_n\right) \\
	&\geq \left[1 - \frac{4\left(1 - q_n^a\right)}{\left(b + l_n\right)q_n^a}\right]\left[1 - \e^{-\frac{n}{\log(n)\sqrt{8\left(b + l_n\right)}}}\right].
	\end{align*}
	
	We conclude that
	\begin{equation*}
	\lim_{n \to \infty} P\left(\left|\calN_n^\alpha\left(u_n\right)\right| \geq \frac{b + l_n}{2}\right) \geq \lim_{n \to \infty} P\left(A_n\right) = 1.
	\end{equation*}
	Therefore, the result follows from the second Borel-Cantelli lemma.
\end{proof}

\subsection{Networks given by simple graphs}
\label{sub: networks given by simple graphs}

Before constructing the example based on Erd\H{o}s-R\'enyi random graphs, we must first explain how to define a network from a graph that is not bipartite. For this purpose, let us consider a simple graph $G = (V, E)$ where each node represents a server that also acts as a dispatcher. We denote the neighborhood of node $u$ by
\begin{equation*}
	\calN(u) \defeq \set{v \in V}{v = u\ \text{or}\ \{u, v\} \in E}.
\end{equation*}
Also, we assume that each task arriving at $u$ is dispatched to a node selected uniformly at random among those in $\calN(u)$ with the least number of tasks. Suppose that tasks arrive at node $u$ as an independent Poisson process of intensity $\lambda(u)$ and that tasks dispatched to $u$ are executed sequentially with independent and exponentially distributed service times of rate $\mu(u)$. Furthermore, denote the number of tasks in node $u$ at time $t$ by $\bX(t, u)$.

\begin{definition}
	\label{def: load balancing process associated with graph}
	We say that $\bX$ is the load balancing process associated with the simple graph $G = (V, E)$ and the rate functions $\map{\lambda}{V}{(0, \infty)}$ and $\map{\mu}{V}{(0, \infty)}$.
\end{definition}

The model introduced above is subsumed by the one described in Section \ref{sec: model description}. Indeed, consider the bipartite graph $G' = (D', S', E')$ such that
\begin{equation*}
	D' \defeq V, \quad S' \defeq V \quad \text{and} \quad E' \defeq \set{(u, v) \in D' \times S'}{u = v\ \text{or}\ \{u, v\} \in E}.
\end{equation*}
Then the neighborhood of node $u$ with respect to the simple graph $G$ is equal to the set of servers that are compatible with dispatcher $u$ with respect to the bipartite graph $G'$. Thus, the load balancing processes associated with $G$ and $G'$ have the same distribution, and in particular Theorem \ref{the: main result} applies to load balancing processes given by Definition \ref{def: load balancing process associated with graph}.

\begin{remark}
	\label{rem: degrees}
	If $\alpha = \left(a, \lambda_{\min}, \mu_{\min}\right) \in \N \times (0, \infty) \times (0, \infty)$, then recall that
	\begin{equation*}
	\calN^\alpha(u) = \set{d \in \calN(u)}{\deg(d) \leq a, \lambda(d) \geq \lambda_{\min}\ \text{and}\ \mu(v) \leq \mu_{\max}\ \text{for all}\ v \in \calN(d)}.
	\end{equation*}
	Here $\deg(d) = \left|\calN(d)\right|$ is relative to the bipartite graph $G'$. However, $\deg(v) = \left|\calN(v)\right| - 1$ if the degree is relative to the graph $G$. Throughout the rest of this section the degree notation refers to $G$ and not to the associated bipartite graph $G'$. Hence,
	\begin{equation*}
	\calN^\alpha(u) = \set{v \in \calN(u)}{\deg(v) \leq a - 1, \lambda(v) \geq \lambda_{\min}\ \text{and}\ \mu(w) \leq \mu_{\max}\ \text{for all}\ w \in \calN(v)}.
	\end{equation*}
\end{remark}

Observe that load balancing processes associated with simple graphs admit a simple sufficient condition for ergodicity. Specifically, if $\lambda(u) < \mu(u)$ for all $u \in V$, then
\begin{equation*}
	\sum_{\calN(u) \subset U} \lambda(u) \leq \sum_{u \in U} \lambda(u) < \sum_{u \in U} \mu(u) \quad \text{for all} \quad \emptyset \neq U \subset S.
\end{equation*}
It follows that condition \eqref{eq: ergodicity condition} holds, and thus $\bX$ is ergodic.

\subsection{Erd\H{o}s-R\'enyi random graphs}
\label{sub: erdos renyi random graphs}

Let $G_n = (V_n, E_n)$ be an Erd\H{o}s-R\'enyi random graph with $n$ nodes and edge probability
\begin{equation*}
p_n \defeq \frac{\log\log(n)}{2(n - 1)}.
\end{equation*}
In particular, the average degree $(n - 1) p_n = \log\log(n) / 2$ approaches infinity slowly. We assume that each node receives tasks at rate $\lambda > 0$ and has processing speed $\mu > \lambda$. Therefore, the associated load balancing process is ergodic.
\begin{theorem}
	\label{the: erdos renyi random graphs}
	Let the random graphs $\set{G_n}{n \geq 1}$ be independent and defined on the same probability space. Fix $\alpha \defeq (2, \lambda, \mu)$ and let $u_n \in S_n$ have maximum degree for each $n \geq 1$. Then with probability one, there exists an infinite sequence $\calM \subset \N$ such that the servers $\set{u_m}{m \in \calM}$ have a skewed neighborhood. Specifically,
	\begin{equation*}
		\lim_{m \to \infty} \left|\calN_m^\alpha\left(u_m\right)\right| = \infty.
	\end{equation*}
\end{theorem}

The sets $\calN_n^\alpha(u)$ only depend on the structure of $G_n$. Indeed, by Remark \ref{rem: degrees},
\begin{equation*}
\calN_n^\alpha\left(u\right) = \set{v \in \calN_n\left(u\right)}{\deg_n(v) \leq 1} \quad \text{for all} \quad u \in S_n.
\end{equation*}
Also, once the graphs $\set{G_n}{n \geq 1}$ have been sampled, each graph defines a load balancing process $\bX_n$. If $X_n$ denotes the stationary distribution of $\bX_n$, then Theorem \ref{the: main result} implies that the servers of maximum degree $\set{u_m}{m \in \calM}$ saturate with $E\left[X_m(u_m)\right] \to \infty$ as $m \to \infty$. However, \cite[Equation (9)]{mukherjee2018asymptotically} implies that the steady-state fraction of servers with more than one task vanishes as $m \to \infty$. Informally speaking, the server with the maximum degree has a diverging number of tasks while nearly all the servers have at most one task.

Consider the random variable and constant defined as
\begin{equation*}
d_n \defeq \deg_n\left(u_n\right) = \max_{u \in S_n} \deg_n(u) \quad \text{and} \quad k_n \defeq \ceil{\frac{1}{2}\left[\log\log(n) + \sqrt{\log(n)\log\log(n)}\right]},
\end{equation*}
respectively; here $u_n \in S_n$ has maximum degree, as assumed in Theorem \ref{the: erdos renyi random graphs}. The first step of the proof of Theorem \ref{the: erdos renyi random graphs} is carried out in the next lemma.

\begin{lemma}
	\label{lem: lower bound for maximum degree of erdos renyi}
	For all sufficiently large $n$, we have
	\begin{equation*}
	P\left(d_n \geq k_n\right) \geq \frac{1}{1 + \sqrt{8 k_n} + \left[\frac{k_n}{(n - 1) p_n}\right]^2}.
	\end{equation*}
\end{lemma}

The proof of the latter lemma and the following are provided in Appendix \ref{app: proofs of various results}.

\begin{lemma}
	\label{lem: second lemma for erdos renyi example}
	For all sufficiently large $n$, we have
	\begin{equation*}
	\cprob*{\left|\calN_n^\alpha(u_n)\right| > \frac{k_n\left(1 - p_n\right)^{n - 2}}{2} | d_n \geq k_n} \geq 1 - \frac{4}{k_n\left(1 - p_n\right)^{n - 2}} - \frac{4p_n}{1 - p_n}.
	\end{equation*}
\end{lemma}

Next we combine the latter lemmas to prove Theorem \ref{the: erdos renyi random graphs}.

\begin{proof}[Proof of Theorem \ref{the: erdos renyi random graphs}]
	Consider the event
	\begin{equation*}
	A_n \defeq \left\{\left|\calN_n^\alpha(u_n)\right| > \frac{k_n(1 - p_n)^{n - 2}}{2}, d_n \geq k_n\right\}.
	\end{equation*}
	It follows from Lemmas \ref{lem: lower bound for maximum degree of erdos renyi} and \ref{lem: second lemma for erdos renyi example} that for all sufficiently large $n$, we have
	\begin{equation*}
	P(A_n) = \cprob*{\left|\calN_n^\alpha(u_n)\right| > \frac{k_n(1 - p_n)^{n - 2}}{2} | d_n \geq k_n} P\left(d_n \geq k_n\right) \geq \frac{1 - \frac{4}{k_n\left(1 - p_n\right)^{n - 2}} - \frac{4p_n}{1 - p_n}}{1 + \sqrt{8 k_n} + \left[\frac{k_n}{(n - 1) p_n}\right]^2}.
	\end{equation*}
	
	By Lemma \ref{lem: limit lemma} of Appendix \ref{app: lemmas used in the examples},
	\begin{equation}
	\label{aux: lower bound involving k n}
	\lim_{n \to \infty} k_n(1 - p_n)^{n - 2} = \lim_{n \to \infty} k_n \e^{-(n - 1)p_n} = \lim_{n \to \infty} \frac{k_n}{\sqrt{\log(n)}} = \infty.
	\end{equation}
	In addition, observe that $k_n / \left[(n - 1) p_n\right]$ behaves asymptotically as
	\begin{equation*}
	\frac{1}{2(n - 1)p_n}\left[\log\log(n) + \sqrt{\log(n)\log\log(n)}\right] = 1 + \sqrt{\frac{\log(n)}{\log\log(n)}}.
	\end{equation*}
	
	We obtain that
	\begin{equation*}
	\sum_{n = 1}^\infty P(A_n) = \infty.
	\end{equation*}
	Thus, the claim follows from the second Borel-Cantelli lemma and \eqref{aux: lower bound involving k n}.
\end{proof}


\begin{appendices}
	
\section{Proofs of various results}
\label{app: proofs of various results}

\begin{proof}[Proof of Corollary \ref{cor: power of d}]
	Consider load balancing processes $\bX_n$ associated with the networks $G_n = (D_n, S_n, E_n)$ and the rate functions $\map{\lambda_n}{D_n}{(0, \infty)}$ and $\map{\mu_n}{S_n}{(0, \infty)}$. We assume that the dispatchers apply a power-of-$d$ policy and the servers $u_n$ have a skewed neighborhood. We will define load balancing processes $\tilde{\bX}_n$ that have the same law as the load balancing processes $\bX_n$ and satisfy the assumptions of Theorem \ref{the: main result}. In particular, the processes $\tilde{\bX}_n$ correspond to dispatchers that assign every incoming task to a server with the least number of tasks among all the compatible servers.
	
	We define $\tilde{G}_n = (\tilde{D}_n, S_n, \tilde{E}_n)$ as follows. Given $e \in D_n$, let $\calS_n(e)$ be the subsets of $\calN_n(e)$ with size $d$; if $e$ is compatible with less than $d$ servers, then we let $\calS_n(e) \defeq \left\{\calN_n(e)\right\}$. Now the sets of dispatchers and compatiblity constraints are defined by:
	\begin{equation*}
	\tilde{D}_n \defeq \set{e_s}{e \in D_n, s \in \calS_n(e)} \quad \text{and} \quad \tilde{E}_n \defeq \bigcup_{e \in D_n} \bigcup_{s \in \calS_n(e)} \left\{e_s\right\} \times s.
	\end{equation*}
	The load balancing process $\tilde{\bX}_n$ is defined by the bipartite graph $\tilde{G}_n$ and the rate functions $\map{\tilde{\lambda}_n}{\tilde{D}_n}{(0, \infty)}$ and $\map{\tilde{\mu}_n}{S_n}{(0, \infty)}$ that are given by
	\begin{equation*}
	\tilde{\mu}_n = \mu_n \quad \text{and} \quad \tilde{\lambda}_n\left(e_s\right) \defeq \frac{\lambda_n(e)}{\left|\calS_n(e)\right|} \quad \text{for all} \quad e \in D_n \quad \text{and} \quad s \in \calS_n(e).
	\end{equation*}
	
	In other words, each dispatcher $e$ is split into $\left|\calS_n(e)\right|$ dispatchers with the same arrival rate, such that each dispatcher assigns tasks to a server with the shortest queue in a different subset of $\calN_n(e)$ of size $d$. Then it is straightforward to check that $\bX_n$ and $\tilde{\bX}_n$ can be coupled in such a way that both processes have the same sample paths. More precisely, this can be done by postulating that a task arrives at $e_s$ for $\tilde{\bX}_n$ if and only if the same task arrives at $e$ for $\bX_n$ and this dispatcher samples the set of servers $s$.
	
	We are assuming that the power-of-$d$ scheme defining $\bX_n$ samples the servers without replacement. However, the proof can be adapted to the case of sampling with replacement. In that case $\calS_n(e)$ are the subsets of $\calN_n(e)$ having size at most $d$, instead of exactly $d$, and $\tilde{\lambda}_n(e_s)$ is defined weighing the probability of sampling $s$, which now depends on $|s|$.
	
	Let us assume that the servers $u_n$ have a skewed neighborhood with respect to the processes $\bX_n$ and a tuple $\alpha = (a, \lambda_{\min}, \mu_{\max})$. We claim that the same servers have a skewed neighborhood with respect to the processes $\tilde{\bX}_n$ and a tuple $\tilde{\alpha}$. Indeed, if $e \in \calN_n^\alpha(u_n)$, then $u_n \in \tilde{\calN}_n(e_s)$ for some $s \in \calS_n(e)$. Further,
	\begin{equation*}
	\tilde{\deg}_n(e_s) \leq \deg_n(e) \leq a \quad \text{and} \quad \tilde{\lambda}_n\left(e_s\right) = \frac{\lambda_n(e)}{\left|\calS_n(e)\right|} \geq \left[\binom{a}{d}\ind{a \geq d} + \ind{a < d}\right]^{-1}\lambda_{\min} \eqdef \tilde{\lambda}_{\min}.
	\end{equation*}
	Moreover, $\tilde{\mu}_n(v) = \mu_n(v) \leq \mu_{\max}$ for all $v \in \tilde{\calN}_n(e_s) \subset \calN_n(e)$. It follows that the servers $u_n$ and the load balancing processes $\tilde{\bX}_n$ satisfy Definition \ref{def: unevenly connected neighborhoods} with $\tilde{\alpha} = \left(a, \tilde{\lambda}_{\min}, \mu_{\max}\right)$. As a result, we conclude from Theorem \ref{the: main result} that \eqref{eq: weak limit is infinity} holds for $\tilde{\bX}_n$, and thus also for $\bX_n$.
\end{proof}

\begin{proof}[Proof of Lemma \ref{lem: occupancy upper bound}]
	Let $G_n' = (D_n', S_n', E_n')$ be the bipartite graph obtained by removing the central servers from the dandelion network $G_n$. Specifically, we define
	\begin{equation*}
		D_n' \defeq D_n, \quad S_n' \defeq \bigcup_{d \in D_n'} B_d \quad \text{and} \quad E_n' \defeq \bigcup_{d \in D_n'} \{d\} \times B_d.
	\end{equation*}
	Let $\bY_n$ be the load balancing process associated with $G_n'$ and the rate functions given by
	\begin{equation*}
		\lambda_n'(d) \defeq \lambda \quad \text{for all} \quad d \in D_n' \quad \text{and} \quad \mu_n'(u) \defeq \mu \quad \text{for all} \quad u \in S_n'.
	\end{equation*}
	
	The proof is carried out by coupling $\bX_n$ and $\bY_n$ in a suitable way. In order to describe this coupling, it is convenient to introduce some notation. For each $x \in \N^{S_n}$ and $d \in D_n$, fix a bijection $\map{\eta(x, d)}{\left\{1, \dots, |B_d|\right\}}{B_d}$ such that $x\left(\eta(x, d, i)\right) \leq x\left(\eta(x, d, j)\right)$ if $i \leq j$. If $x$ represents the occupancies of the servers, then this function arranges the servers in $B_d$ in a way that is monotone with respect to the number of tasks at each server. We fix functions $\eta(y, d)$ with analogous properties for all $y \in \N^{S_n'}$ and $d \in D_n'$.
	
	The coupling is as follows. We postulate that all servers are initially empty:
	\begin{equation*}
		\bX_n(0, u) = 0 \quad \text{for all} \quad u \in S_n \quad \text{and} \quad \bY_n(0, u) = 0 \quad \text{for all} \quad u \in S_n'.
	\end{equation*}
	Moreover, each dispatcher $d \in D_n = D_n'$ has the same arrival process for $\bX_n$ and $\bY_n$, and we determine the departures from the servers in $B_d$ using $|B_d|$ common potential departure processes; these are independent Poisson processes of rate $\mu$. Specifically, the potential departure processes are indexed by $\left\{1, \dots, |B_d|\right\}$ and a jump of process $i$ at time $t$ corresponds to a potential departure from servers $\eta\left(\bX_n\left(t^-\right), d, i\right)$ and $\eta\left(\bY_n\left(t^-\right), d, i\right)$ for $\bX_n$ and $\bY_n$, respectively; a potential departure from a server leads to an actual departure from the server if the server has a positive number of tasks. We let the potential departure processes of the central servers of $\bX_n$ be independent of everything else.
	
	We claim that the latter coupling leads to
	\begin{equation}
		\label{aux: inequality between x and y}
		\vec{\bX}_n(t, d, i) \defeq \bX_n\left(t, \eta\left(\bX_n\left(t\right), d, i\right)\right) \leq \bY_n\left(t, \eta\left(\bY_n\left(t\right), d, i\right)\right) \eqdef \vec{\bY}_n(t, d, i)
	\end{equation}
	for each sample path and for all $t \geq 0$, $d \in D_n = D_n'$ and $i \in \left\{1, \dots, |B_d|\right\}$. The inequality clearly holds at time zero and is preserved at each potential departure by definition of the coupling. It is also preserved at each arrival since dispatcher $d$ sends each task to a server with the least number of tasks in $B_d \cup C$ for $\bX_n$ and $B_d$ for $\bY_n$. Since $\bX_n$ and $\bY_n$ are constant between arrival and potential departure times, we conclude by induction on the latter times that the inequality holds at all times.
	
	It follows from \eqref{aux: inequality between x and y} that
	\begin{equation*}
		\ind{\vec{\bX}_n(t, d, i) \geq k} \leq \ind{\vec{\bY}_n(t, d, i) \geq k}
	\end{equation*}
	for each sample path and for all $t \geq 0$, $d \in D_n = D_n'$, $i \in \left\{1, \dots, |B_d|\right\}$ and $k \in \N$. Thus, the ergodicity of $\bX_n$ and $\bY_n$ implies that
	\begin{equation*}
		P\left(\vec{X}_n(d, i) \geq k \right) \leq P\left(\vec{Y}_n(d, i) \geq k \right),
	\end{equation*}
	where $X_n$ is the stationary distribution of $\bX_n$, $Y_n$ is the stationary distribution of $\bY_n$, $\vec{X}_n(d, i) \defeq X_n\left(\eta\left(X_n, d, i\right)\right)$ and $\vec{Y}_n(d, i) \defeq Y_n\left(\eta\left(Y_n, d, i\right)\right)$. This implies that
	\begin{equation*}
		E\left[\vec{X}_n(d, i)\right] = \sum_{k = 1}^\infty P\left(\vec{X}_n(d, i) \geq k \right) \leq \sum_{k = 1}^\infty P\left(\vec{Y}_n(d, i) \geq k \right) = E\left[\vec{Y}_n(d, i)\right].
	\end{equation*}
	
	Property (a) follows from the latter inequality. Indeed,
	\begin{equation*}
		P\left(\sum_{d \in D} \sum_{b \in B_d} \left|X_n(b)\right| \geq k\right) \leq \frac{1}{k}\sum_{d \in D} \sum_{b \in B_d} E\left[X_n(b)\right] \leq \frac{1}{k}\sum_{d \in D} \sum_{b \in B_d} E\left[Y_n(b)\right] \quad \text{for all} \quad k > 0.
	\end{equation*}
	Note that $E\left[Y_n(b)\right]$ is finite and independent of $n$ since $\vct{Y_n(u)}{u \in B_d}$ is the stationary distribution of the same basic load balancing process for all $d \in D$ and $n \geq \max D$. Therefore, the above inequality implies tightness as in (a).
	
	In order to obtain property (b), observe that
	\begin{align*}
		\frac{1}{\left|D_n\right|}\sum_{d \in D_n} \frac{1}{\mu}E\left[\min\set{X_n(u)}{u \in B_d \cup C}\right] &\leq \frac{1}{\left|D_n\right|}\sum_{d \in D_n} \frac{1}{\mu}E\left[\vec{X}_n(d, 1)\right] \\
		&\leq \frac{1}{\left|D_n\right|}\sum_{d \in D_n} \frac{1}{\mu}E\left[\vec{Y}_n(d, 1)\right].
	\end{align*}
	The left and right sides are the mean waiting time of an incoming task for $X_n$ and $Y_n$, respectively. It follows from Little's law that the average total number of tasks is smaller for $X_n$ than for $Y_n$. Because the mean total number of tasks is finite for $Y_n$, we conclude that it is also finite for $X_n$, and thus (b) holds.
\end{proof}

\begin{proof}[Proof of Lemma \ref{lem: coupling of dispatching decisions}]
	It follows from \eqref{aux: monotonicity before arrival} that
	\begin{equation*}
		\min_{u \in \calN_1(d)} x_1(u) \geq \min_{u \in \calN_1(d)} x_2\left(\varphi_d(u)\right) \geq \min_{u \in \calN_2(d)} x_2(u).
	\end{equation*}
	If the first inequality is strict, then
	\begin{equation*}
		x_1(u) \geq x_2\left(\varphi_d(u)\right) + 1 \quad \text{for all} \quad u \in \calN_1(d) \quad \text{such that} \quad \varphi_d(u) \in \calM_2(d, x_2).
	\end{equation*}
	Moreover, $\varphi_d\left(\calN_1(d)\right) \cap \calM_2(d, x_2) = \emptyset$ if the second inequality is strict. In either case \eqref{aux: monotonicity after arrival} holds if we take independent random variables $U_1$ and $U_2$ that are uniformly distributed over $\calM_1(d, x_1)$ and $\calM_2(d, x_2)$, respectively.
	
	Therefore, it only remains to consider the case where
	\begin{equation*}
		\min_{u \in \calN_1(d)} x_1(u) = \min_{u \in \calN_2(d)} x_2(u).
	\end{equation*}
	In this case $\varphi_d\left(\calM_1(d, x_1)\right) \subset \calM_2(d, x_2)$ by \eqref{aux: monotonicity before arrival}. Indeed, $u \in \calM_1(d, x_1)$ implies that
	\begin{equation*}
		x_2\left(\varphi_d(u)\right) \leq x_1(u) = \min_{v \in \calN_1(d)} x_1(v) = \min_{v \in \calN_2(d)} x_2(v).
	\end{equation*}
	
	Let $U$ and $U_2$ be uniform in $\calM_1(d, x_1)$ and $\calM_2(d, x_2)$, respectively, and define
	\begin{equation*}
		U_1 \defeq \begin{cases}
			\varphi_d^{-1}\left(U_2\right) & \text{if} \quad U_2 \in \varphi_d\left(\calM_1(d, x_1)\right), \\
			U & \text{if} \quad U_2 \notin \varphi_d\left(\calM_1(d, x_1)\right).
		\end{cases}
	\end{equation*}
	It is straightforward to check that $U_1$ is uniformly distributed over $\calM_1(d, x_1)$. Furthermore, if $u \in \calN_1(d)$ and $U_2 = \varphi_d(u)$, then we must have
	\begin{equation*}
		u \in \calM_1(d, x_1) \quad \text{and} \quad U_1 = u, \quad \text{or} \quad u \notin \calM_1(d, x_1) \quad \text{and} \quad x_1(u) \geq x_2\left(\varphi_d(u)\right) + 1;
	\end{equation*}
	for the latter inequality note that $\varphi_d(u) = U_2 \in \calM_2(d, x_2)$. Hence, \eqref{aux: monotonicity after arrival} holds.
\end{proof}

\begin{proof}[Proof of Lemma \ref{lem: lower bound for maximum degree of erdos renyi}]
	Let
	\begin{equation*}
		Y_n \defeq \left|\set{u \in V_n}{\deg_n(u) \geq k_n}\right|
	\end{equation*}
	be the number of nodes with degree at least $k_n$. By \cite[Lemmas 2 and 3]{bollobas1981degree},
	\begin{equation*}
		\mu_n \defeq E\left[Y_n\right] = nP\left(Z_n^1 \geq k_n\right) \quad \text{and} \quad \sigma_n^2 \defeq \var*{Y_n} \leq \mu_n + \left[nP\left(Z_n^2 = k_n - 1\right)\right]^2,
	\end{equation*}
	with $Z_n^1 \sim \text{Bin}(n - 1, p_n)$ and $Z_n^2 \sim \text{Bin}(n - 2, p_n)$ binomially distributed. It follows from the second-moment method that
	\begin{equation*}
		P\left(d_n \geq k_n\right) = P\left(Y_n > 0\right) \geq \frac{\mu_n^2}{\mu_n^2 + \sigma_n^2}\geq \frac{1}{1 + \frac{1}{\mu_n} + \left[\frac{P\left(Z_n^2 = k_n - 1\right)}{P\left(Z_n^1 \geq k_n\right)}\right]^2}.
	\end{equation*}
	
	Now observe that
	\begin{equation*}
		\frac{P\left(Z_n^2 = k_n - 1\right)}{P\left(Z_n^1 \geq k_n\right)} \leq \frac{\binom{n - 2}{k_n - 1} p_n^{k_n - 1}(1 - p_n)^{n - k_n - 1}}{\binom{n - 1}{k_n} p_n^{k_n} (1 - p_n)^{n - k_n - 1}} = \frac{k_n}{(n - 1) p_n}.
	\end{equation*}
	By Lemma \ref{lem: lower bound for binomial} of Appendix \ref{app: lemmas used in the examples}, we also have
	\begin{equation*}
		\mu_n = n P\left(Z_n^1 \geq k_n\right) = nP\left(Z_n^1 \geq (n - 1)p_n + l_n\right) \geq \frac{1}{\sqrt{8\left[(n - 1)p_n + l_n\right]}} \geq \frac{1}{\sqrt{8k_n}}
	\end{equation*}
	for all large enough $n$, where $l_n \defeq \sqrt{(n - 1)p_n \log(n) / 2}$. Therefore, we obtain
	\begin{equation*}
		P\left(d_n \geq k_n\right) \geq \frac{1}{1 + \sqrt{8k_n} + \left[\frac{k_n}{(n - 1) p_n}\right]^2}
	\end{equation*}
	for all sufficiently large $n$.
\end{proof}

\begin{proof}[Proof of Lemma \ref{lem: second lemma for erdos renyi example}] Let $I_n(u) \defeq \ind{\deg_n(u) = 1}$. If $u, v \neq u_n$ and $u \neq v$, then we let:
	\begin{align*}
		&\mu_n \defeq \expect*{I_n(u) | u \in \calN_n(u_n)} = \left(1 - p_n\right)^{n - 2}, \\
		&\sigma_n^2 \defeq \var*{I_n(u) | u \in \calN_n(u_n)} = \left(1 - p_n\right)^{n - 2} - (1 - p_n)^{2(n - 2)}, \\
		&\rho_n \defeq \cov*{I_n(u), I_n(v) | u, v \in \calN_n(u_n)} = \left(1 - p_n\right)^{2n -5} - \left(1 - p_n\right)^{2(n - 2)}.
	\end{align*}
	For the last equality note that both $u$ and $v$ can have at most $n - 3$ neighbors distinct from $u_n$, $u$ and $v$. If the only neighbor of both $u$ and $v$ is $u_n$, then the corresponding $2n - 6$ edges, and the edge between $u$ and $v$, must be absent.
	
	Recall that $\calN_n^\alpha(u_n) = \set{u \in \calN_n(u)}{\deg_n(u) \leq 1}$. Hence,
	\begin{equation*}
		\left|\calN_n^\alpha(u_n)\right| = \ind{d_n \leq 1} + \Gamma_n, \quad \text{with} \quad \Gamma_n \defeq \sum_{u \in \calN_n(u_n)\setminus\{u_n\}} I_n(u);
	\end{equation*}
	the first indicator accounts for the possibility that $u_n \in \calN_n^\alpha(u_n)$. Since we are interested in the distribution of $\left|\calN_n^\alpha(u_n)\right|$ given that $d_n \geq k_n$, we may focus on the term $\Gamma_n$. Note that $\expect*{\Gamma_n | d_n} = d_n \mu_n$. Furthermore,
	\begin{align*}
		\var*{\Gamma_n | d_n} &= d_n\sigma_n^2 + d_n(d_n - 1)\rho_n \\
		&= d_n \mu_n - d_n(1 - p_n)^{2(n - 2)} \\
		&+ d_n(d_n - 1)(1 - p_n)^{2n - 5} - d_n(d_n - 1)(1 - p_n)^{2(n - 2)} \\
		&= d_n \mu_n + d_n^2(1 - p_n)^{2n - 5} - d_n^2(1 - p_n)^{2(n - 2)} - d_n(1 - p_n)^{2n - 5} \\
		&\leq d_n \mu_n + d_n^2(1 - p_n)^{2n - 5}p_n;
	\end{align*}
	for the last step, note that $(1 - p_n)^{2n - 5} - (1 - p_n)^{2(n - 2)} = (1 - p_n)^{2n - 5}p_n$. Then
	\begin{equation*}
		\cprob*{\Gamma_n \leq \frac{d_n \mu_n}{2} | d_n} \leq \cprob*{\left|\Gamma_n - d_n\mu_n\right| \geq \frac{d_n \mu_n}{2} | d_n} \leq \frac{4\left[d_n \mu_n + d_n^2(1 - p_n)^{2n - 5}p_n\right]}{\left(d_n \mu_n\right)^2}.
	\end{equation*}
	
	If $k_n \geq 2$, then we conclude that
	\begin{align*}
		\cprob*{\left|\calN_n^\alpha(u_n)\right| > \frac{k_n \mu_n}{2} | d_n \geq k_n} &= \cprob*{\Gamma_n > \frac{k_n \mu_n}{2} | d_n \geq k_n}\\
		&\geq \cprob*{\Gamma_n > \frac{d_n \mu_n}{2} | d_n \geq k_n} \geq 1 - \frac{4}{k_n \mu_n} - \frac{4p_n}{1 - p_n}.
	\end{align*}
	This completes the proof.
\end{proof}

\section{Lemmas used in the examples}
\label{app: lemmas used in the examples}

The following lemma provides a lower bound for the tail of a binomial distribution. A more general lower bound can be found in \cite[Lemma 4.7.2]{ash2012information}.

\begin{lemma}
	\label{lem: lower bound for binomial}
	Fix $\set{k_n \geq 1, p_n \in (0, 1)}{n \geq 1}$ such that
	\begin{equation*}
		\lim_{n \to \infty} p_n = 0, \quad \lim_{n \to \infty} \frac{p_n\log(k_n)}{n} = 0 \quad \text{and} \quad \liminf_{n \to \infty} \frac{\log(k_n)}{n p_n} > 0.
	\end{equation*}
	If $Z_n \sim \mathrm{Bin}(n, p_n)$ is the sum of $n$ independent Bernoulli random variables of mean $p_n$, then the following inequality holds for all large enough $n$:
	\begin{equation*}
		P\left(Z_n \geq n p_n + l_n\right) \geq \frac{1}{k_n\sqrt{8\left(np_n + l_n\right)}} \quad \text{with} \quad l_n \defeq \sqrt{\frac{n p_n \log(k_n)}{2}}. 
	\end{equation*}
\end{lemma}

\begin{proof}
	Let
	\begin{equation*}
		D(a||b) \defeq a\log\left(\frac{a}{b}\right) + (1 - a)\log\left(\frac{1 - a}{1 - b}\right)
	\end{equation*}
	be the Kullback-Leibler divergence between Bernoulli random variables of mean $a$ and $b$, respectively. It follows from \cite[Lemma 4.7.2]{ash2012information} that
	\begin{equation}
		\label{aux: inequality from ash}
		P\left(Z_n \geq n\theta\right) \geq \frac{1}{\sqrt{8n\theta\left(1 - \theta\right)}}\e^{-nD\left(\theta || p_n\right)} \quad \text{for all} \quad p_n < \theta < 1.
	\end{equation}
	
	Consider a sequence $\set{\alpha_n > 0}{n \geq 1}$ such that
	\begin{equation*}
		\liminf_{n \to \infty} \alpha_n > 1 \quad \text{and} \quad \lim_{n \to \infty} \alpha_n p_n = 0.
	\end{equation*}
	This implies that \eqref{aux: inequality from ash} holds, for all large enough $n$, with $\theta$ replaced by $\alpha_n p_n$. Further,
	\begin{equation*}
		\alpha_n - \frac{1 - \alpha_np_n}{1 - p_n} \leq 2\left(\alpha_n - 1\right)
	\end{equation*}
	for all sufficiently large $n$. Since $\log(x) \leq x - 1$ for all $x > 0$,
	\begin{align*}
		D\left(\alpha_n p_n || p_n\right) &= \alpha_n p_n \log(\alpha_n) + (1 - \alpha_n p_n) \log\left(\frac{1 - \alpha_n p_n}{1 - p_n}\right) \\
		&\leq \alpha_n p_n (\alpha_n - 1) + (1 - \alpha_n p_n)\left(\frac{1 - \alpha_n p_n}{1 - p_n} - 1\right) \\
		&= \alpha_n p_n (\alpha_n - 1) + (1 - \alpha_n)p_n\left(\frac{1 - \alpha_n p_n}{1 - p_n}\right) \\
		&= (\alpha_n - 1)\left(\alpha_n - \frac{1 - \alpha_n p_n}{1 - p_n}\right) p_n \leq 2(\alpha_n - 1)^2p_n.
	\end{align*}
	As a result, the following inequality holds for all large enough $n$:
	\begin{equation*}
		P\left(Z_n \geq n \alpha_n p_n\right) \geq \frac{1}{\sqrt{8n \alpha_n p_n\left(1 - \alpha_n p_n\right)}}\e^{-2n(\alpha_n - 1)^2 p_n}.
	\end{equation*}
	
	It is straightforward to check that
	\begin{equation*}
		\liminf_{n \to \infty} \frac{n p_n + l_n}{n p_n} > 1 \quad \text{and} \quad \lim_{n \to \infty} \frac{n p_n + l_n}{n} = 0.
	\end{equation*}
	Now the claim follows by letting $\alpha_n \defeq (np_n + l_n) / (n p_n)$, because
	\begin{align*}
	\frac{1}{\sqrt{8n \alpha_n p_n\left(1 - \alpha_n p_n\right)}}\e^{-2n(\alpha_n - 1)^2 p_n} &= \frac{1}{\sqrt{8\left(n p_n + l_n\right)\left(1 - \alpha_np_n\right)}}\e^{-\log(k_n)} \geq \frac{1}{k_n\sqrt{8\left(n p_n + l_n\right)}}.
	\end{align*} 
	This completes the proof.
\end{proof}

The other lemma that we need for constructing the examples of Section \ref{sec: random networks} is below.

\begin{lemma}
	\label{lem: limit lemma}
	If $\set{p_n \geq 0}{n \geq 1}$ is such that
	\begin{equation*}
		\lim_{n \to \infty} n p_n^2 = 0, \quad \text{then} \quad \lim_{n \to \infty} \frac{(1 - p_n)^n}{\e^{-np_n}} = 1.
	\end{equation*}
\end{lemma}

\begin{proof}
	For each $y < 0$, there exists $\theta(y) \in (y, 0)$ such that
	\begin{equation*}
		\e^y = 1 + y + \frac{\e^{\theta(y)} y^2}{2} \leq 1 + y + \frac{y^2}{2}.
	\end{equation*}
	
	Suppose that $0 \leq x \leq 1 / 2$. Then
	\begin{align*}
		\e^{-x - 2x^2} &\leq 1 - x - 2x^2 + \frac{\left(x + 2x^2\right)^2}{2} \\
		&= 1 - x - 2x^2 + \frac{x^2 + 4x^3 + 4x^4}{2} = 1 - x + \left(\frac{1}{2} + 2x + 2x^2 - 2\right)x^2 \leq 1 - x.
	\end{align*}
	Since $1 - x \leq \e^{-x}$ for all $x \in \R$, we obtain
	\begin{equation*}
		\e^{-\left(p_n + 2p_n^2\right)n} \leq (1 - p_n)^n \leq \e^{-np_n} \quad \text{whenever} \quad p_n \leq \frac{1}{2}.
	\end{equation*}
	By assumption $p_n \to 0$ and $\e^{-2np_n^2} \to 1$ as $n \to \infty$, which completes the proof.
\end{proof}

\end{appendices}
	
\newcommand{\noop}[1]{}
\bibliographystyle{IEEEtranS}
\bibliography{bibliography}

\begin{thebibliography}{10}
\providecommand{\url}[1]{#1}
\csname url@samestyle\endcsname
\providecommand{\newblock}{\relax}
\providecommand{\bibinfo}[2]{#2}
\providecommand{\BIBentrySTDinterwordspacing}{\spaceskip=0pt\relax}
\providecommand{\BIBentryALTinterwordstretchfactor}{4}
\providecommand{\BIBentryALTinterwordspacing}{\spaceskip=\fontdimen2\font plus
\BIBentryALTinterwordstretchfactor\fontdimen3\font minus
  \fontdimen4\font\relax}
\providecommand{\BIBforeignlanguage}[2]{{%
\expandafter\ifx\csname l@#1\endcsname\relax
\typeout{** WARNING: IEEEtranS.bst: No hyphenation pattern has been}%
\typeout{** loaded for the language `#1'. Using the pattern for}%
\typeout{** the default language instead.}%
\else
\language=\csname l@#1\endcsname
\fi
#2}}
\providecommand{\BIBdecl}{\relax}
\BIBdecl

\bibitem{ash2012information}
R.~B. Ash, \emph{Information theory}.\hskip 1em plus 0.5em minus 0.4em\relax
  Dover, 1990.

\bibitem{bahi2003broken}
J.~Bahi, R.~Couturier, and F.~Vernier, ``Broken edges and dimension exchange
  algorithms on hypercube topology,'' in \emph{Eleventh Euromicro Conference on
  Parallel, Distributed and Network-Based Processing}.\hskip 1em plus 0.5em
  minus 0.4em\relax IEEE, 2003, pp. 140--145.

\bibitem{bollobas1981degree}
B.~Bollob{\'a}s, ``Degree sequences of random graphs,'' \emph{Discrete
  Mathematics}, vol.~33, no.~1, pp. 1--19, 1981.

\bibitem{van2018scalable}
M.~{\noop{Boor}}van~der Boor, S.~C. Borst, J.~S.~H. van Leeuwaarden, and
  D.~Mukherjee, ``Scalable load balancing in networked systems: A survey of
  recent advances,'' \emph{SIAM Review}, vol.~64, no.~3, pp. 554--622, 2022.

\bibitem{bramson2011stability}
M.~Bramson, ``Stability of join the shortest queue networks,'' \emph{The Annals
  of Applied Probability}, vol.~21, no.~4, pp. 1568--1625, 2011.

\bibitem{braverman2016stein}
A.~Braverman, J.~Dai, and J.~Feng, ``Stein’s method for steady-state
  diffusion approximations: an introduction through the {E}rlang-{A} and
  {E}rlang-{C} models,'' \emph{Stochastic Systems}, vol.~6, no.~2, pp.
  301--366, 2016.

\bibitem{budhiraja2019supermarket}
A.~Budhiraja, D.~Mukherjee, and R.~Wu, ``Supermarket model on graphs,''
  \emph{The Annals of Applied Probability}, vol.~29, no.~3, pp. 1740--1777,
  2019.

\bibitem{cardinaels2019job}
E.~Cardinaels, S.~C. Borst, and J.~S.~H. van Leeuwaarden, ``Job assignment in
  large-scale service systems with affinity relations,'' \emph{Queueing
  Systems}, vol.~93, pp. 227--268, 2019.

\bibitem{cruise2020stability}
J.~Cruise, M.~Jonckheere, and S.~Shneer, ``Stability of {JSQ} in queues with
  general server-job class compatibilities,'' \emph{Queueing Systems}, vol.~95,
  pp. 271--279, 2020.

\bibitem{cybenko1989dynamic}
G.~Cybenko, ``Dynamic load balancing for distributed memory multiprocessors,''
  \emph{Journal of parallel and distributed computing}, vol.~7, no.~2, pp.
  279--301, 1989.

\bibitem{elsasser2004load}
R.~Elsasser, B.~Monien, and S.~Schamberger, ``Load balancing in dynamic
  networks,'' in \emph{7th International Symposium on Parallel Architectures,
  Algorithms and Networks.}\hskip 1em plus 0.5em minus 0.4em\relax IEEE, 2004,
  pp. 193--200.

\bibitem{eryilmaz2012asymptotically}
A.~Eryilmaz and R.~Srikant, ``Asymptotically tight steady-state queue length
  bounds implied by drift conditions,'' \emph{Queueing Systems}, vol.~72, pp.
  311--359, 2012.

\bibitem{ethier2009markov}
S.~N. Ethier and T.~G. Kurtz, \emph{Markov processes: characterization and
  convergence}.\hskip 1em plus 0.5em minus 0.4em\relax John Wiley \& Sons,
  2009.

\bibitem{foss1998stability}
S.~G. Foss and N.~I. Chernova, ``On the stability of a partially accessible
  multi-station queue with state-dependent routing,'' \emph{Queueing Systems},
  vol.~29, pp. 55--73, 1998.

\bibitem{gast2015power}
N.~Gast, ``The power of two choices on graphs: the pair-approximation is
  accurate?'' \emph{ACM SIGMETRICS Performance Evaluation Review}, vol.~43,
  no.~2, pp. 69--71, 2015.

\bibitem{gast2017refined}
N.~Gast and B.~Van~Houdt, ``A refined mean field approximation,''
  \emph{Proceedings of the ACM on Measurement and Analysis of Computing
  Systems}, vol.~1, no.~2, pp. 1--28, 2017.

\bibitem{gilbert2021complexity}
S.~Gilbert, U.~Meir, A.~Paz, and G.~Schwartzman, ``On the complexity of load
  balancing in dynamic networks,'' in \emph{Proceedings of the 33rd ACM
  Symposium on Parallelism in Algorithms and Architectures}, 2021, pp.
  254--264.

\bibitem{glynn2008bounding}
P.~W. Glynn and A.~Zeevi, ``Bounding stationary expectations of {M}arkov
  processes,'' in \emph{Markov processes and related topics: a Festschrift for
  Thomas G. Kurtz}.\hskip 1em plus 0.5em minus 0.4em\relax Institute of
  Mathematical Statistics, 2008, vol.~4, pp. 195--214.

\bibitem{goldsztajn2023sparse}
D.~Goldsztajn, S.~C. Borst, and J.~S.~H. van Leeuwaarden, ``Load balancing with
  sparse dynamic random graphs,'' \emph{arXiv preprint arXiv:2305.13054}, 2023.

\bibitem{gurvich2014diffusion}
I.~Gurvich, ``Diffusion models and steady-state approximations for
  exponentially ergodic {M}arkovian queues,'' \emph{The Annals of Applied
  Probability}, vol.~24, no.~6, pp. 2527--2559, 2014.

\bibitem{kenthapadi2006balanced}
K.~Kenthapadi and R.~Panigrahy, ``Balanced allocation on graphs,'' in
  \emph{SODA}, vol.~6, 2006, pp. 434--443.

\bibitem{liu2020steady}
X.~Liu and L.~Ying, ``Steady-state analysis of load-balancing algorithms in the
  sub-{H}alfin--{W}hitt regime,'' \emph{Journal of Applied Probability},
  vol.~57, no.~2, pp. 578--596, 2020.

\bibitem{maguluri2016heavy}
S.~T. Maguluri and R.~Srikant, ``Heavy traffic queue length behavior in a
  switch under the maxweight algorithm,'' \emph{Stochastic Systems}, vol.~6,
  no.~1, pp. 211--250, 2016.

\bibitem{menich1991optimality}
R.~Menich and R.~F. Serfozo, ``Optimality of routing and servicing in dependent
  parallel processing systems,'' \emph{Queueing Systems}, vol.~9, pp. 403--418,
  1991.

\bibitem{mitzenmacher2001power}
M.~Mitzenmacher, ``The power of two choices in randomized load balancing,''
  \emph{IEEE Transactions on Parallel and Distributed Systems}, vol.~12,
  no.~10, pp. 1094--1104, 2001.

\bibitem{mukherjee2018asymptotically}
D.~Mukherjee, S.~C. Borst, and J.~S.~H. van Leeuwaarden, ``Asymptotically
  optimal load balancing topologies,'' \emph{Proceedings of the ACM on
  Measurement and Analysis of Computing Systems}, vol.~2, no.~1, pp. 1--29,
  2018.

\bibitem{nygren2010akamai}
E.~Nygren, R.~K. Sitaraman, and J.~Sun, ``The {A}kamai network: a platform for
  high-performance internet applications,'' \emph{ACM SIGOPS Operating Systems
  Review}, vol.~44, no.~3, pp. 2--19, 2010.

\bibitem{rutten2022load}
D.~Rutten and D.~Mukherjee, ``Load balancing under strict compatibility
  constraints,'' \emph{Mathematics of Operations Research}, vol.~48, no.~1, pp.
  227--256, 2023.

\bibitem{rutten2023meanSHORT}
------, ``Mean-field analysis for load balancing on spatial graphs,'' \emph{ACM
  SIGMETRICS Performance Evaluation Review}, vol.~51, no.~1, pp. 27--28, 2023.

\bibitem{rutten2023mean}
------, ``Mean-field analysis for load balancing on spatial graphs,''
  \emph{arXiv preprint arXiv:2301.03493}, 2023.

\bibitem{sparaggis1993extremal}
P.~D. Sparaggis, D.~Towsley, and C.~Cassandras, ``Extremal properties of the
  shortest/longest non-full queue policies in finite-capacity systems with
  state-dependent service rates,'' \emph{Journal of Applied Probability},
  vol.~30, pp. 223--236, 1993.

\bibitem{tang2019random}
D.~Tang and V.~G. Subramanian, ``Random walk based sampling for load balancing
  in multi-server systems,'' \emph{Proceedings of the ACM on Measurement and
  Analysis of Computing Systems}, vol.~3, no.~1, pp. 1--44, 2019.

\bibitem{turner1998effect}
S.~R. Turner, ``The effect of increasing routing choice on resource pooling,''
  \emph{Probability in the Engineering and Informational Sciences}, vol.~12,
  pp. 109--124, 1998.

\bibitem{vvedenskaya1996queueing}
N.~D. Vvedenskaya, R.~L. Dobrushin, and F.~I. Karpelevich, ``Queueing system
  with selection of the shortest of two queues: An asymptotic approach,''
  \emph{Problemy Peredachi Informatsii}, vol.~32, no.~1, pp. 20--34, 1996.

\bibitem{wang2018heavy}
W.~Wang, S.~T. Maguluri, R.~Srikant, and L.~Ying, ``Heavy-traffic delay
  insensitivity in connection-level models of data transfer with proportionally
  fair bandwidth sharing,'' \emph{ACM SIGMETRICS Performance Evaluation
  Review}, vol.~45, no.~3, pp. 232--245, 2018.

\bibitem{wang2014maptask}
W.~Wang, K.~Zhu, L.~Ying, J.~Tan, and L.~Zhang, ``{MapTask} scheduling in
  {MapReduce} with data locality: Throughput and heavy-traffic optimality,''
  \emph{IEEE/ACM Transactions on Networking}, vol.~24, no.~1, pp. 190--203,
  2016.

\bibitem{weng2020optimal}
W.~Weng, X.~Zhou, and R.~Srikant, ``Optimal load balancing in bipartite
  graphs,'' \emph{arXiv preprint arXiv:2008.08830}, 2020.

\bibitem{weng2020boptimal}
------, ``Optimal load balancing with locality constraints,'' \emph{Proceedings
  of the ACM on Measurement and Analysis of Computing Systems}, vol.~4, no.~3,
  pp. 1--37, 2020.

\bibitem{wieder2017hashing}
U.~Wieder, ``Hashing, load balancing and multiple choice,'' \emph{Foundations
  and Trends{\textregistered} in Theoretical Computer Science}, vol.~12, no.
  3--4, pp. 275--379, 2017.

\bibitem{xie2015priority}
Q.~Xie and Y.~Lu, ``Priority algorithm for near-data scheduling: Throughput and
  heavy-traffic optimality,'' in \emph{2015 IEEE International Conference on
  Computer Communications (INFOCOM)}.\hskip 1em plus 0.5em minus 0.4em\relax
  IEEE, 2015, pp. 963--972.

\bibitem{xie2016scheduling}
Q.~Xie, A.~Yekkehkhany, and Y.~Lu, ``Scheduling with multi-level data locality:
  Throughput and heavy-traffic optimality,'' in \emph{2016 IEEE International
  Conference on Computer Communications (INFOCOM)}.\hskip 1em plus 0.5em minus
  0.4em\relax IEEE, 2016, pp. 1--9.

\bibitem{ying2017stein}
L.~Ying, ``Stein's method for mean field approximations in light and heavy
  traffic regimes,'' \emph{Proceedings of the ACM on Measurement and Analysis
  of Computing Systems}, vol.~1, no.~1, pp. 1--27, 2017.

\bibitem{zhao2023optimal}
Z.~Zhao and D.~Mukherjee, ``Optimal rate-matrix pruning for large-scale
  heterogeneous systems,'' \emph{arXiv preprint arXiv:2306.00274}, 2023.

\bibitem{zhao2022exploiting}
Z.~Zhao, D.~Mukherjee, and R.~Wu, ``Exploiting data locality to improve
  performance of heterogeneous server clusters,'' \emph{arXiv preprint
  arXiv:2211.16416}, 2022.

\end{thebibliography}
	
\end{document}